\documentclass[leqno]{siamltex}
\usepackage{amsfonts}
\usepackage{amsmath, amssymb}
\usepackage{color}
\usepackage{graphicx}
\usepackage[bookmarks,colorlinks=true,citecolor=blue]{hyperref}
\usepackage{enumerate}
\usepackage{bbm}

\usepackage{tikz}
\usetikzlibrary{calc,patterns,arrows}

\usepackage{todonotes}

\DeclareMathOperator{\Div}{div}
\DeclareMathOperator{\vspan}{span}

\def\eps{\varepsilon}
\def\TH{{\mathcal{T}_{H}}}
\def\R{\mathbb{R}}
\newcommand{\Z}{\mathbb{Z}}

\newtheorem{remark}[theorem]{Remark}
\newtheorem{assumption}[theorem]{Assumption}

\newcommand{\VHext}{{V}_H^{ext}}

\newcommand{\VH}{{V}_H}
\newcommand{\ZHext}{{Z}_H^{ext}}

\newcommand{\ZH}{{Z}_H}
\newcommand{\espHC}{H^2 (\Omega) \cap C^1 (\bar\Omega)}
\newcommand{\average}[1]{\langle#1\rangle}

\newcommand{\figref}[1]{Fig.~\ref{#1}}

\mathchardef\ordinarycolon\mathcode`\:
\mathcode`\:=\string"8000
\begingroup \catcode`\:=\active
  \gdef:{\mathrel{\mathop\ordinarycolon}}
\endgroup

\pgfdeclarepatternformonly{XCircles}{\pgfqpoint{-1pt}{-1pt}}{\pgfqpoint{1cm}{1cm}}{\pgfqpoint{0.5cm}{0.5cm}}
{
  \pgfsetlinewidth{0.4pt}
  \pgfpathellipse{\pgfqpoint{0.53cm}{0.54cm}}{\pgfqpoint{0cm}{0.1cm}}{\pgfqpoint{0.1cm}{0cm}}
  \pgfusepath{stroke}
}

\pgfdeclarepatternformonly[\XHatchSize]{XHatch}{\pgfqpoint{-1pt}{-1pt}}{\pgfqpoint{\XHatchSize}{\XHatchSize}}{\pgfqpoint{\XHatchSize}{\XHatchSize}}
{
  \pgfsetlinewidth{0.2pt}
  \pgfpathmoveto{\pgfqpoint{0pt}{0pt}}
  \pgfpathlineto{\pgfqpoint{\XHatchSize}{\XHatchSize}}
  \pgfusepath{stroke}
}

\newdimen\XHatchSize
\tikzset{
    XHatchSize/.code={\XHatchSize=#1},
    XHatchSize=10pt
}

\title{Non-Conforming Multiscale Finite Element Method for Stokes Flows in
Heterogeneous Media. Part II: error estimates for periodic microstructure }
\author{G. Jankowiak\footnotemark[1] \and A. Lozinski\footnotemark[2] }
\date{August 2015}

\begin{document}

\maketitle

\begin{abstract}
This paper is dedicated to the rigorous numerical analysis of a Multiscale Finite Element Method (MsFEM) for the Stokes system,
    when dealing with highly heterogeneous media,
    as proposed in {\it B.P.~Muljadi et al.,  Non-conforming multiscale finite Element method for Stokes flows in heterogeneous media. Part I: Methodologies and numerical experiments}, SIAM MMS (2015), \textbf{13}(4) {1146-–1172}. The method is in the vein of the
    classical Crouzeix-Raviart approach. It is  generalized here to arbitrary sets of weighting functions used to enforced continuity across the mesh edges. We provide error bounds for a particular set of weighting functions in a periodic setting, using an accurate estimate of the homogenization error.    Numerical experiments demonstrate an improved accuracy of the present variant with respect to that of Part I, both  in the periodic case and in a broader setting.
\end{abstract}



\renewcommand{\thefootnote}{\fnsymbol{footnote}}
\footnotetext[1]{\href{mailto:gaspard.jankowiak@math.cnrs.fr}{gaspard.jankowiak@math.cnrs.fr}. 
WPI - Universit\"{a}t Wien, 
1090 Vienna, Austria.}
\footnotetext[2]{\href{mailto:alexei.lozinski@univ-fcomte.fr}{alexei.lozinski@univ-fcomte.fr}. 
Laboratoire de Math\'ematiques, UMR CNRS 6623, Univ. Bourgogne
Franche-Comt\'e, 25030 Besan\c{c}on, France.}
\renewcommand{\thefootnote}{\arabic{footnote}}

\section{Introduction}
\label{sec:intro}

We consider the Stokes problem in the perforated domain $\Omega^\varepsilon:=\Omega\setminus B^{\varepsilon}$, $\Omega\subset\R^2$: find $u:\Omega^\varepsilon\rightarrow\mathbb{R}^2$ and $
p:\Omega^\varepsilon\rightarrow\mathbb{R}$, solution of
\begin{alignat}{2}
-\Delta u+\nabla p &= f&\quad&\text{ on }\Omega^\varepsilon\,,  \label{genP} \\
\Div u &= 0&&\text{ on }\Omega^\varepsilon\,,  \label{genD} \\
u &= 0&&\text{ on }\partial\Omega^\varepsilon\,,  \label{genB}
\end{alignat}
where $f:\Omega\rightarrow\mathbb{R}^2$ is a given function, assumed
sufficiently regular on $\Omega$. 

We are interested in the situations where the perforations $B^\eps$ have a complex structure, making a direct numerical solution  of problem (\ref{genP})--(\ref{genB}) very expensive. Typically,  $B^\eps$ is assumed to be a set of obstacles of average size and average inter-obstacle distance $\eps<<diam(\Omega)$, so that the mesh resolving all the features of the perforated domain $\Omega^\eps$ is too complex. Our goal is to devise an efficient numerical method that employs a relatively coarse mesh of size $H\ge\eps$ (or even  $H>>\eps$). We borrow the concept of Multiscale Finite Element Method (MsFEM) \cite{thomhou, houefendiev}, where the multiscale basis functions are pre-calculated on each cell of the coarse mesh, using a local sufficiently fine mesh, to represent a typical behavior of the microscopic structure of the flow. The global approximation to the solution of the problem in $\Omega^\eps$ is then constructed as the Galerkin projection on the space spanned by these basis functions.

The particular variant of MsFEM pursued in this article is inspired by classical non-conforming Crouzeix-Raviart finite elements \cite{CRRairo}. The idea of Crouzeix-Raviart MsFEM was first developed in \cite{MsFEMCR1, MsFEMCR2} for diffusion problems either with highly oscillating coefficients or posed on a perforated domain. It was also extended to advection-diffusion problems in \cite{muljadicrmsfem} and to Stokes equation in the Part I of the present series of papers \cite{StokesPart1}. In the construction of Crouzeix-Raviart multiscale basis functions, the conformity between coarse elements is not enforced in a strong sense. The basis functions are required to be continuous only in a weak (finite element) sense, i.e. merely the averages of the jumps of these functions vanish at coarse element edges. The boundary conditions at the edges are then provided by  a natural decomposition of the entire functional space into the sum of unresolved fine scales and the finite set of multiscale base functions. In the present article, we generalize this idea by introducing the weights into the averages over the edges in the definition of the functional spaces. This additional flexibility allows us to construct a more accurate variant of Crouzeix-Raviart MsFEM, as confirmed by the numerical experiments at the end of this article. Moreover, we are now able to provide a rigorous {\it a priori} error bounds in terms of $H$ and $\eps$ in a periodic setting, i.e. when $B^\eps$ is populated by the same pattern repeated periodically on a grid of size $\eps$.


%

Let us mention briefly other approaches which can be applied to similar problems:  wavelet-based homogenization method \cite{dorobantuengquist}, variational multiscale method \cite{nolenetal}, equation-free computations \cite{kevrekidisetal}, heterogeneous multiscale method \cite{weinanengquist} and many others. For viscous, incompressible flows, multiscale methods based on homogenization theory for solving slowly varying Stokes flow in porous media have been studied in \cite{brownetal, brownefendievhoang}. Returning to the MsFEM-type approaches, we should mention a big amount of work on the oversampling approach, first introduced in the original work \cite{thomhou} to provide a better approximation of the edge boundary condition of the multiscale basis functions. Oversampling here means that the local problem in the coarse element are extended to a domain larger than the element itself, but only the interior information would be communicated to the coarse scale equation. Various extensions of the sampled domain lead to various oversampling methods, cf. \cite{houefendiev, chuetal, henningpeterseim, efendievetal}. The Crouzeix-Raviart MsFEM considered here does not require oversampling. A numerical comparison Crouzeix-Raviart MsFEM with oversampling one was performed in \cite{MsFEMCR1, MsFEMCR2} for diffusion problems. It revealed that both methods yield at least qualitatively the same results, while Crouzeix-Raviart MsFEM outperforms all the other variants in the tests on perforated domains.

This paper is organized as follows. The Crouzeix-Raviart MsFEM is presented in section \ref{sec:crmsfem}. We recall there namely the construction from Part I \cite{StokesPart1} and explain and motivate some modifications and generalization we make to this construction here. We also announce there the main theoretical result of the paper: an a priori error bound in the case of periodic perforations. The rest of the paper (with the exception of some numerical experiments) is constrained to this periodic setting. Section \ref{sec:homog} deals with the homogenization theory. We prove there an estimate of the error committed by the approximation of the Stokes equations with the Darcy ones. Section \ref{sec:preliminaries} deals with some technical lemmas. Section \ref{sec:proof} presents the proof of the MsFEM error bound. Finally, the numerical tests are reported in Section \ref{sec:numres}.

\section{MsFEM \emph{\`{a} la} Crouzeix-Raviart}\label{sec:crmsfem}

We assume henceforth that $\Omega$ is a polygonal domain. We define a mesh
$\mathcal{T}_{H}$ on $\Omega$, \emph{i.e.} a decomposition of $\Omega$ into
polygons, each of diameter at most $H$, and denote $\mathcal{E}_{H}$ the set
of all the edges of $\mathcal{T}_{H}$, $\mathcal{E}_{H}^{int}\subset\mathcal{E}_H$ the internal edges and $\mathcal E(T)$ the
set of edges of~$T\in\mathcal T_H$. Note that we mesh $\Omega$
and not the perforated domain $\Omega^\varepsilon$. This allows us to use
coarse elements (independently of the fine scales present in the geometry of $
\Omega^\varepsilon$) and leaves us with a lot of flexibility: some
mesh nodes may be in $B^{\varepsilon}$, and likewise some edges may intersect $
B^{\varepsilon}$.

We assume that the mesh does not have any hanging nodes, \emph{i.e.} each internal edge is shared by exactly two mesh cells.
In addition, $\mathcal{T}_{H}$ is assumed to be quasi-uniform in the following
sense: fixing a polygon~$\overline{T} \subset \mathbb \R^2$ as reference element (one can also have a finite collection of reference elements), for any mesh element $T\in\mathcal{T}_{H}$, there exists a smooth invertible mapping $K:\overline{T}\to T$ such that $\|\nabla
K\|_{L^{\infty}}\leq CH$, $\|\nabla K^{-1}\|_{L^{\infty}}\leq CH^{-1}$, $C$
being some universal constant independent of~$T$, which we will refer to as the
regularity parameter of the mesh. To avoid some technical complications, we
also assume that the mappings $K$ are affine on every edge of
$\partial\overline{T}$. These assumptions are obviously met by a triangular mesh satisfying the minimum angle condition (see \emph{e.g.}~\cite[Section~4.4]{brenner}), but our approach carries over to quadrangles, which are in
fact used for our numerical computations, or to general polygonal meshes (in the flavor of Virtual Finite Elements \cite{VirtualFE}).

We shall use the usual notations $L^2(\omega)$, $H^k(\omega)$ for Sobolev spaces on a domain $\omega$. We shall also denote $L^2_0(\omega) = \left\{ p\in L^2(\omega) : \int_{\omega} p = 0 \right\}$ and $H^1_0(\omega) = \left\{ u\in H^1(\omega) : u_{\partial\omega} = 0 \right\}$. We shall implicitly identify the functions in $H^1_0(\Omega^\eps)$ with those in $H^1(\Omega)$ vanishing on $B^\eps$. The weak form of (\ref{genP})--(\ref{genB}) can be written as: find $(u,p)\in H^1_0(\Omega^\eps)^2\times L^2_0(\Omega^\eps)$ such that
$$
c((u,p),(v,q))=\int_\Omega f\cdot v,\quad\forall (v,q)\in H^1_0(\Omega^\eps)^2\times L^2_0(\Omega^\eps)
$$
with
\begin{equation}\label{cform}
c\left( (u, p), (v, q) \right) := \int_{\Omega^\varepsilon}
\nabla u : \nabla v - \int_{\Omega^\varepsilon} p \Div v
- \int_{\Omega^\varepsilon} q \Div u\,.
\end{equation}

We shall also need the broken Sobolev spaces of the type $H^1(\mathcal{T}_H)=\{{u}\in L^2(\Omega): {u}|_{T}\in H^{1}(T)^2 \text{ for any }T\in \mathcal{T}_{H}$. To make notations simpler, the integrals over $\Omega$ or $\Omega^\eps$ involving such functions will be implicitly split into the sums of the integrals over the mesh cells
$$
\int_{\Omega}\bullet =\sum_{T\in \mathcal{T}_{H}}\int_{T}\bullet
\qquad\text{and}\qquad
\int_{\Omega^{\varepsilon }}\bullet =\sum_{T\in \mathcal{T}_{H}}\int_{\Omega
^{\varepsilon }\cap T}\bullet \,.
$$
The same convention shall be implicitly assumed in the notation of the $H^1$ semi-norm, i.e. we define $|u|_{H^1(\Omega)}=\left(\sum_{T\in\mathcal{T}_H}|u|_{H^1(\Omega)^2}\right)^{1/2}$ for a function $u$ in the broken $H^1$ space. 

The idea of the Multiscale Finite Element Method (MsFEM)
\emph{\`{a} la} Crouzeix-Raviart is to require the continuity of the finite
element functions, which here are highly oscillatory, in the sense of some
weighted averages on the edges. We have adapted this approach to the Stokes equation in~\cite{StokesPart1} 
using the simplest possible set of weights on the edges. We are now going to recall the main ideas of
this construction and to generalize it to arbitrary weighting functions.

\subsection{Functional spaces}
Let us fix a positive integer $s$ and associate some vector-valued functions $\omega_{E,1},\ldots,\omega_{E,s}:E\to\R^2$
to any edge $E\in\mathcal{E}_H$.  As in \cite{StokesPart1}, we first introduce the extended velocity space 
\begin{equation*}
V_H^{ext}:=\left\{
\begin{array}{l}
{u}\in L^2(\Omega)^2 \text{ such that }{u}|_{T}\in
H^{1}(T)^2 \text{ for any }T\in \mathcal{T}_{H},
\\ {u}=0\text{ on } B^{\varepsilon }, \text{ and }
\int_{E}[[u]]\cdot\omega_{E,j}=0\text{ for all }E\in \mathcal{E}_{H}, \ j=1,\ldots,s
\end{array}
\right\} ,
\end{equation*}
where $[[u]]$ denotes the jump of $u$ across an internal edge and $[[u]] = u$ on the boundary
$\partial\Omega$. The idea behind this space is to enhance the natural velocity space $H^1_0 (\Omega^\eps)^2$
so that we have at our disposal the vector fields discontinuous across the edges of
the mesh. Indeed, our aim is to construct a nonconforming approximation method, where
the continuity of the solution on the mesh edges will be preserved only for the weighted averages. We shall need some technical requirements on the weights:

\begin{assumption}\label{AssWeights} 
For any $E\in\mathcal{E}_H$,  $\vspan (\omega_{E,1}\ldots,\omega_{E,s}) \supset n_E$, the unit normal to $E$.
\end{assumption}

Note that the original construction from \cite{StokesPart1} is recovered by setting the weights as $s=2$, 
$\omega_{E,1}=e_1=\left(\begin{array}{c}  1\\  0\end{array}\right)$, 
$\omega_{E,2}=e_2=\left(\begin{array}{c}  0\\  1\end{array}\right)$ on all the edges. Assumption \ref{AssWeights} is then trivially verified. This will be also the case for another choice of the weights introduced later in this article.

The following assumptions deal not only with the weights but also with the manner in which the holes $B_\eps$ intersect the mesh cells.
\begin{assumption}\label{AssWeights2} 
Take any $T\in\mathcal{T}_H$ and any real numbers $c^E_1,\ldots,c^E_s$ on all the edges $E$ composing $\partial T$. There exists $v\in H^1(T)^2$ vanishing on $T\cap B^\eps$ and such that $\int_E v\cdot\omega_{E,i}=c^E_i$,  $i=1,\ldots,s$ for all the edges $E$.  
\end{assumption}
\begin{assumption}\label{AssWeights3} 
For any $T\in\TH$, let $C_1,\ldots,C_n$ be the connected components of $T\cap\Omega^\eps$  and choose any real numbers $c_1,\ldots,c_n$ with $\sum_{i=1}^nc_i=0$. There exists $w\in H^1(T)^2$ vanishing on $T\cap B^\eps$ and such that $\int_{\partial C_i} w\cdot n=c_i$, $i=1,\ldots,n$ and $\int_F w\cdot\omega_{F,j}=0$ for all the edges $F$ of $T$ and $j=1,\ldots,s$.  
\end{assumption}

\begin{remark}
 Assumption \ref{AssWeights2} above will be valid provided the weights $\omega_{E,1},\ldots,\omega_{E,s}$ are linearly independent, and $E$ is not covered completely by $B^\eps$. Note that the situations where some edges $E$ are  covered  by $B^\eps$ can be easily handled by a slight modification of the fourth-coming MsFEM  method, cf. Lemma \ref{PhiGen}): one should simply ignore such edges when constructing the MsFEM basis functions. 
 
Assumption \ref{AssWeights3} on the other hand may impose some restrictions on the choice of the mesh with respect to the perforations. However, it will be satisfied in most typical situations. First of all, we emphasize that this Assumption is void if $T\cap\Omega^\eps$ is connected (one puts then $w=0$). Moreover, the required function $w$ can be easily constructed if, for example, a mesh element $T$ is split by $B^\eps$ into two connected components $C_1$, $C_2$ and one of its edges, say $E$, is split into two non-empty connected components, say $E_{C_1}$, $E_{C_2}$: one can prescribe then proper non-zero averages of normal fluxes of $w$ on $E_{C_1}$ and $E_{C_2}$  while letting them equal to 0 on all the other edges composing the boundary of $C_1$ and $C_2$. Similar constructions can be imagined in other more complicated situations.  
\end{remark}

We introduce now the combined velocity-pressure space~$X_H^{ext} = V_H^{ext} \times M$,
with~$M = L^2_0(\Omega^\varepsilon) := \left\{ p\in L^2(\Omega^\varepsilon) \text{ s.t.} \int_{\Omega^\varepsilon} p = 0 \right\}$.
The space $X_H^{ext}$ is then decomposed into coarse and fine components:
\begin{equation}\label{XHdirect}
 X_H^{ext} = X_H \oplus^{\perp_c} X_H^0
\end{equation}
where $X_H^0 = V_H^0 \times M_H^0$ is the space of unresolved fine scales with
\begin{align*}     
    V_H^0 &:= \left\{u\in V_H^{ext} : \int_E u\cdot\omega_{E,j} = 0\quad \forall E\in\mathcal E_H,\ j=1,\ldots,s \right\} \\
    M_H^0 &:=\left\{p\in M : \int_{T\cap\Omega^\eps} p = 0\quad \forall T\in\mathcal T\right\}
\end{align*}
and $X_H$ is chosen as the orthogonal complement of $X_H^0$ with respect to~$c$,
the natural bilinear form (\ref{cform}) associated to the Stokes problem. The space $X_H$ in (\ref{XHdirect}) will be referred to as the Crouzeix-Raviart MsFEM space and used to construct an approximation method.    

The basis function of $X_H$ can be constructed in a localized manner, i.e. their supports cover a small number of mesh cells, as in the standard finite element shape functions. We summarize this construction in the following 
\begin{lemma}\label{PhiGen}
Define the functional spaces $M_{H}\subset M$ and $V_{H}\subset V_{H}^{ext}$ as
\begin{align}
\label{MHdef}
M_{H}&=\{{q}\in L^2_{0}(\Omega )\,\text{\ such that }\,q|_{T}=const,~%
\forall T\in \mathcal{T}_{H}\}
\\
\label{VHdef}
V_{H}&=\vspan\{{\Phi}_{E,i},~E\in \mathcal{E}_{H}^{int},~i=1,\ldots,s\}.
\end{align}
where 	${\Phi}_{E,i}$ for any $E\in \mathcal{E}_{H}^{int}$, $i=1,\ldots,s$ is the vector-valued function on $\Omega$, vanishing outside 
the two mesh cells $T_{1},T_{2}$ adjacent to $E$, and defined on these two cells together with the accompanying pressure $\pi
_{E,i}$  as the solution to the following problems: for $k=1,2$, find  ${ \Phi}_{E,i}\in H^1(T_k)^2$ s.t. ${ \Phi}_{E,i}=0$ on $T_k\cap B^{\varepsilon}$,
 $\pi_{E,i}\in L^2_0(T_k\cap \Omega^\varepsilon)$, and the Lagrange multipliers ${\lambda}_{F,j}\in\mathbb{R}$ for all~$F\in\mathcal{E}(T_k)$, $j=1,\ldots,s$ such that
\begin{multline}
\int_{T_k\cap\Omega^{\varepsilon }}\nabla { \Phi}_{E,i}:\nabla {v}
-\int_{T_k\cap\Omega^{\varepsilon }}\pi_{E,i}\mathop{\rm div}\nolimits {v}
-\int_{T_k\cap\Omega^{\varepsilon }}q\Div \Phi_{E,i}
\label{FBaseWeak} \\
+\sum_{j=1}^s\sum_{F\in \mathcal{E}(T_k)}\left[{\lambda}_{F,j} \int_{F}v\cdot\omega_{F,j}
+{\mu}_{F,j} \int_{F} \Phi_{E,i}\cdot\omega_{F,j}\right]
=
\mu_{E,i}
\end{multline}
for all ${v}\in H^{1}(T_k\cap\Omega^{\varepsilon})$ s.t.~$ v|_{T_k\cap B^\varepsilon} = 0$,
$q\in L_{0}^{2}(T_k\cap\Omega^{\varepsilon })$, ${\mu}_{F,j}\in \mathbb{R}$ for all $F\in \mathcal{E}(T_k)$, $j=1,\ldots,s)$. 

Under Assumptions \ref{AssWeights}--\ref{AssWeights3} the problems above are well posed and the MsFEM space $X_H$ from (\ref{XHdirect}) can be identified with
\begin{equation}
X_{H}=\vspan\{({u}_{H},\pi _{H}({u}_{H})+\bar{p}_{H}),~{u}_{H}\in V_{H},~\bar{p}_{H}\in M_{H}\}  \label{XHdef1}
\end{equation}%
where $\pi_H:V_H\to \vspan\{{\pi}_{E,i},~E\in \mathcal{E}_{H}^{int},~i=1,2\}\subset M_H^0$ is the linear mapping such that $\pi_H({\Phi}_{E,i})={\pi}_{E,i}$ for all $E\in \mathcal{E}_{H}^{int}$, $i=1,2$.
\end{lemma}

\begin{remark}
 In the strong form, problem (\ref{FBaseWeak}) can be rewritten as: find ${\Phi}_{E,i}$ and ${\pi}_{E,i}$ that  solve on $T_k$, $k=1,2$
\begin{align*}
-\Delta {\Phi}_{E,i}+\nabla \pi _{E,i} &=0,& &\text{ on }\Omega
^{\varepsilon }\cap T_{k}, \\
\Div {\Phi}_{E,i} &=const, & &\text{ on }\Omega ^{\varepsilon }\cap
T_{k}, \\
{\Phi}_{E,i} &=0, & &\text{ on }B^{\varepsilon }\cap T_{k}, \\
\nabla {\Phi}_{E,i}n-\pi _{E,i}n &\in \vspan\{\omega_{F,1},\ldots,\omega_{F,s}\} & & \text{ on }F\cap \Omega
^{\varepsilon },\quad\text{ for all }F\in \mathcal{E}(T_{k}), \\
\int_{F}{\Phi}_{E,i}\cdot\omega_{F,j} &=\left\{ 
\begin{array}{c}
\delta_{ij},~F=E \\ 
0,~F\not=E%
\end{array}%
\right. & &\text{ for all }F\in \mathcal{E}(T_{k}), j=1,\ldots,s\\
\int_{\Omega ^{\varepsilon }\cap T_{k}}\pi _{E,i} &=0. & &\quad
\end{align*}
\end{remark}

\begin{proof}
The well-posedness of problem (\ref{FBaseWeak}) on any mesh element $T_k$ (denoted simply by $T$ in the sequel of this proof) follows from Assumption \ref{AssWeights2}, which ensures that one can prescribe the needed values to $\int_{F}{\Phi}_{E,i}\cdot\omega_{F,j}$ on all the edges $F$ of element $T_k$, and from the following inf-sup condition
\begin{equation}\label{infsupEl}
\inf_{q\in L^2_{0}(T\cap \Omega ^{\varepsilon })}\sup_{{v}\in V_{\int
0}(T)}\frac{\int_{T\cap \Omega ^{\varepsilon }}q\Div{v}}{\Vert
q\Vert _{L^2 (T)}|{v}|_{H^{1}(T)}}>0
\end{equation}
with
\begin{equation*}
V_{\int 0}(T)=\{{v}:H^{1}(T)^2:\int_{E}{v}=0,~\forall E\in 
\mathcal{E}(T)\text{ and }{v}=0\text{ on }B^{\varepsilon }\cap T\}.
\end{equation*}
In turn, property (\ref{infsupEl}) can be established thanks to Assumption \ref{AssWeights3}. Indeed, this property is evident if $T\cap \Omega ^{\varepsilon }$ is connected: given $q\in L^2_{0}(T\cap \Omega ^{\varepsilon })$ one takes then $v\in H^1_0(T\cap \Omega ^{\varepsilon })\subset V_{\int 0}(T)$ (assuming that $v$ is extended by zero on $B^\eps$) such that $\Div v=q$ and the $H^1$ norm of $v$ is bounded by the $L^2$ norm of $q$ (the existence of such a function is assured by \cite[Corollary 2.4, p.24]{GiraultRaviart}). If not, recall the connected components $C_1,\ldots,C_n$ of $T\cap\Omega^\eps$, denote $c_i=\int_{C_i}q$ for a given $q\in L^2_{0}(T\cap \Omega ^{\varepsilon })$, observe $\sum_{i=1}^nc_i=0$ and consider the function $w$ from Assumption \ref{AssWeights3}. We have $w\in V_{\int 0}(T)$ and
$$
\int_{C_i}\Div w=\int_{\partial C_i}w\cdot n=c_i=\int_{C_i}q
$$
One can now choose $w^{(i)}\in H^1_0(C_i)$ on each component $C_i$ such that $\Div w^{(i)}=q-\Div w$ on $C_i$. Such functions exist thanks to the above mentioned result from \cite{GiraultRaviart} since $\int_{C_i}(q-\Div w)=0$.  Setting $v=w+w^{(i)}$ on $C_i$ and $v=0$ on $B^\eps$ gives $v\in V_{\int 0}(T)$ such that $\Div v=q$. By construction, the $H^1$ norm of $v$ is bounded by the $L^2$ norm of $q$.

To prove the other statements of the proposition one can easily adapt the proofs of Lemmas 3.1, 3.2 and Remark 3.3 in \cite{StokesPart1} with obvious modifications induced by the more general constraints $\int_{E}[[u]]\cdot \omega_{E,j}=0$, replacing $\int_{E}[[u]]=0$ in the definition of $V_{H}^{ext}$. We shall not go into the details of these modifications for the sake of brevity. We emphasize only that Assumption \ref{AssWeights} is indeed necessary to conclude. For example, the proof that $\int_{\Omega ^{\varepsilon }}\bar{p}_{H}\Div{v}=0$ for any $\bar{p}_H\in M_H$ and ${v}\in V_{H}^{0}$, cf. Lemma 3.1 in \cite{StokesPart1},  goes like this
\[
\int_{\Omega ^{\varepsilon }}\bar{p}_{H}\Div{v}
=\sum_{T\in\mathcal{T}_{H}}\bar{p}_{H}|_{T}\int_{T}\Div{v}
=\sum_{T\in \mathcal{T}_{H}}\bar{p}_{H}|_{T}\int_{\partial T}{v}%
\cdot n=0 
\]%
The last equality above is justified by $\vspan (\omega_{E,1}\ldots,\omega_{E,s}) \supset n_E$ on any edge $E$ of $T$.
\end{proof}

From now on, we can think of $V_H$ as the finite dimensional space defined by (\ref{VHdef}). From this construction, we see easily that $\Div(V_H)\subset M_H$. Indeed, $\Div(v_H)$ is piecewise constant on $\TH$ for any $v_H \in V_H$ and
$$
\int_{\Omega^\eps}\Div v_H=\sum_{T\in\TH}\int_T \Div v_H=\sum_{E\in\mathcal{E}_H}\int_E [[v_H\cdot n_E]]=0 .
$$
In fact, $\Div(V_H)=M_H$ as will be shown in Lemma \ref{infsupD}.

\subsection{The MsFEM approximation}
The approximation of the solution to the Stokes problem (\ref{genP})--(\ref{genB}) now reads: find $
{u}_{H}\in \VH{}$ and $p_{H}\in M_{H}$ such that
\begin{alignat}{2}
\int_{\Omega ^{\varepsilon }}\nabla {u}_{H}:\nabla {v}
_{H}-\int_{\Omega ^{\varepsilon }}p_{H}\Div{v}_{H} &=\int_{\Omega
^{\varepsilon }}{f}\cdot {v}_{H}&\quad& \forall {v}_{H}\in \VH{}\,,
\label{MsFEMu} \\
\int_{\Omega ^{\varepsilon }}q_{H}\Div{u}_{H} &=0&& \forall
q_{H}\in M_{H}\,,  \label{MsFEMp}
\end{alignat}

Existence and uniqueness of the solution to \eqref{MsFEMu}--\eqref{MsFEMp}
follows from the standard theory of saddle-point problems provided the pair of spaces $V_H\times M_H$ satisfies the inf-sup property. This is indeed the case, as shown in the next lemma.
\begin{lemma}\label{infsupD}
Assume that the continuous velocity-pressure inf-sup property holds on $\Omega^\eps$ with a constant $\beta>0$, i.e.
$$
\inf_{p\in L^2_0(\Omega^\eps)}\sup_{v\in H^1_0{\Omega^\eps}^2}\frac{\int_{\Omega ^{\varepsilon
}}p\Div{v}}{\Vert p\Vert_{L^2}|{v}|_{H^{1}}}\geq \beta
$$
Then, the discrete inf-sup property holds on $V_H\times M_H$ with the same constant $\beta >0$:
\begin{equation*}
\inf_{p_{H}\in M_{H}}\sup_{v_{H}\in \VH{}}\frac{\int_{\Omega ^{\varepsilon
}}p_{H}\Div{v}_{H}}{\Vert p_{H}\Vert_{L^2}|{v}_{H}|_{H^{1}}}\geq \beta\,.
\end{equation*}
More precisely, for any $p_H\in M_H$ there exists $v_H\in\VH$ such that
$$  
\Div v_H=p_H \text{ on }T\cap\Omega^\eps, \ \forall T\in\TH
  \quad\text{and}\quad
  {|v_H |_{H^1(\Omega)}} \le \frac{1}{\beta} \|p_H\|_{L^2(\Omega^\eps)} \,. 
$$
\end{lemma}
\begin{proof}
Take arbitrary $p_{H}\in M_{H}$ and $v\in H^1_0(\Omega^\eps)$ such that 
$$  
\Div v=p_H \text{ on }\Omega^\eps  \quad\text{and}\quad
  {|v|_{H^1(\Omega)}} \le \frac{1}{\beta} \|p_H\|_{L^2(\Omega^\eps)} \,. 
$$
Decompose $v=v_H+v_H^0$ with $v_H\in V_H$ and $v_H\in V_H^0$. It implies $\int_E v\cdot n_E=\int_E v_H\cdot n_E$ on any $E\in\mathcal{E}_H$ so that for any $T\in\TH$
\begin{equation*}
\int_{T\cap\Omega^\eps} \Div v_H = \int_{\partial T} v_H\cdot n = \int_{\partial T}  v\cdot n = \int_{T\cap\Omega^\eps} \Div v =
\int_{T\cap\Omega^\eps} p_H.
\end{equation*}
Since, both $\Div v_H$ and $p_H$ are piecewise constant on $\TH$, we conclude  $\Div v_H=p_H$.

Moreover, $\int_{\Omega^\eps} \nabla(v-v_H):w_H=0$ for any $w_H\in V_H$ by the construction of $V_H$, cf. the orthogonality between $X_H$ and $X_H^0$. Hence $|{v}_{H}|_{H^{1}}\le|{v}|_{H^{1}}$ which proves the Lemma.
\end{proof}

In fact, the velocity $u_H$ given by (\ref{MsFEMu})--(\ref{MsFEMp}) can be characterized in a simpler manner: find ${u}_{H}\in \ZH$ such that
\begin{equation}
\int_{\Omega ^{\varepsilon }}\nabla {u}_{H}:\nabla {v}
_{H}=\int_{\Omega ^{\varepsilon }}{f}\cdot {v}_{H},\quad \forall
{v}_{H}\in \ZH
\label{MsFEMz}
\end{equation}
where $Z_H$ is the divergence free subspace of $V_H$:
\begin{equation}\label{ZHdef}
\ZH =\{{u_H}\in \VH\text{ such that }\Div {u_H}=0\text{
on any }T\in \mathcal{T}_{H}\}\,.
\end{equation}
This fact will be useful in the proof of the error estimate.

\begin{remark}\label{BCg}
Our method can be easily adapted to non-homogeneous boundary conditions on the outer boundary $\partial\Omega$, i.e. when (\ref{genB}) is replaced with 
\begin{equation}
u= g\text{ on }\partial\Omega \quad\text{and}\quad u= 0\text{ on }\partial B^\eps \,.
\label{genBg}
\end{equation}
One should then add the following equations on all the mesh edges $E$ lying on $\partial\Omega$:
$$
\int_E u_H \cdot \omega_{E,j} = \int_E g \cdot \omega_{E,j},\quad j=1,\ldots,s.
$$
\end{remark}

\subsection{Possible choices of weighting functions}
We now consider two choices of weighting functions, leading to 2 variants of multiscale spaces:
\begin{align}
\label{CR2}
\text{CR}_2 & :  s=2, \ \omega_{E,1}=e_1, \ \omega_{E,2}=e_2, \\
\label{CR3}
\text{CR}_3 & :  s=3, \ \omega_{E,1}=e_1, \ \omega_{E,2}=e_2, \ \omega_{E,3}=n_E\psi_E, 
\end{align}
for any $E\in \mathcal{E}_{H}$. Here $n_E$ denote again a unit vector normal to $E$ and $\psi_{E}$
a linear polynomial on $E$ such that $\int_{E}\psi_{E}=0$. The actual choice of $n_{E}$ and $\psi_{E}$ should be made once for all, but is arbitrary otherwise.

We recognize the space CR$_2$ as being the MsFEM space from the Part I of the present series \cite{StokesPart1} where it was successfully tested numerically. It can however be quite inefficient in certain situations, especially when some of the mesh cells contain a lot of densely packed holes. Consider, for example, a geometrical configuration as in Fig. \ref{fig:base}. We represent there a mesh cell (a square), say $T\in\TH$, which happens to contain 9 round holes. We plot on the left the sum of basis functions $\Phi_{LR}:=\Phi_{L,1}+\Phi_{R,1}$ from the CR$_2$ basis associated to the two vertical sides of $T$, $L$ being the left side and $R$ the right side of $T$ and assuming that the unit normal is chosen in the direction $e_1$ on both edges $L$ and $R$. We thus impose the flow to be (in average) in $e_1$ direction on both vertical sides of $T$ and to vanish (in average again) on both horizontal sides. We consider a sum of basis functions, rather than a basis function alone, since $\Div(\Phi_{LR})=0$ on $T$ as $\int_T\Div\Phi_{LR}=\int_{R}\Phi_{R,1}\cdot e_1-\int_{L,1}\Phi_{L}\cdot e_1=0$, while $\int_T\Div\Phi_{L,1}=const\not=0$ there. The vector field $\Phi_{LR}$ should model, roughly speaking, the flow from left to right inside $T$. However, the actual behavior of $\Phi_{LR}$ is quite different and counter-intuitive: the fluid seems to turn around the corners of the cell $T$, which have of course no physical meaning, and barely penetrates inside $T$ between the obstacles. One concludes thus that the CR$_2$ space $V_H$ cannot be used in general to construct a reasonable approximation of the solution to the Stokes problem.   Turning to the alternative CR$_3$ space, we plot at Fig. \ref{fig:base} on the right the same linear combination $\Phi_{L,1}+\Phi_{R,1}$ of basis functions. We see now that their behavior is at least visually correct.  The superiority of CR$_3$ over CR$_2$ will be further confirmed by other numerical experiments in Section \ref{sec:numres}. Moreover, we shall be able to prove an error estimate for the MsFEM approximation using the CR$_3$ basis functions, cf. Theorem \ref{theo:main} below and its proof in Section~\ref{sec:proof}.

\begin{figure}[h]
    \centering
    \includegraphics[width=0.4\textwidth]{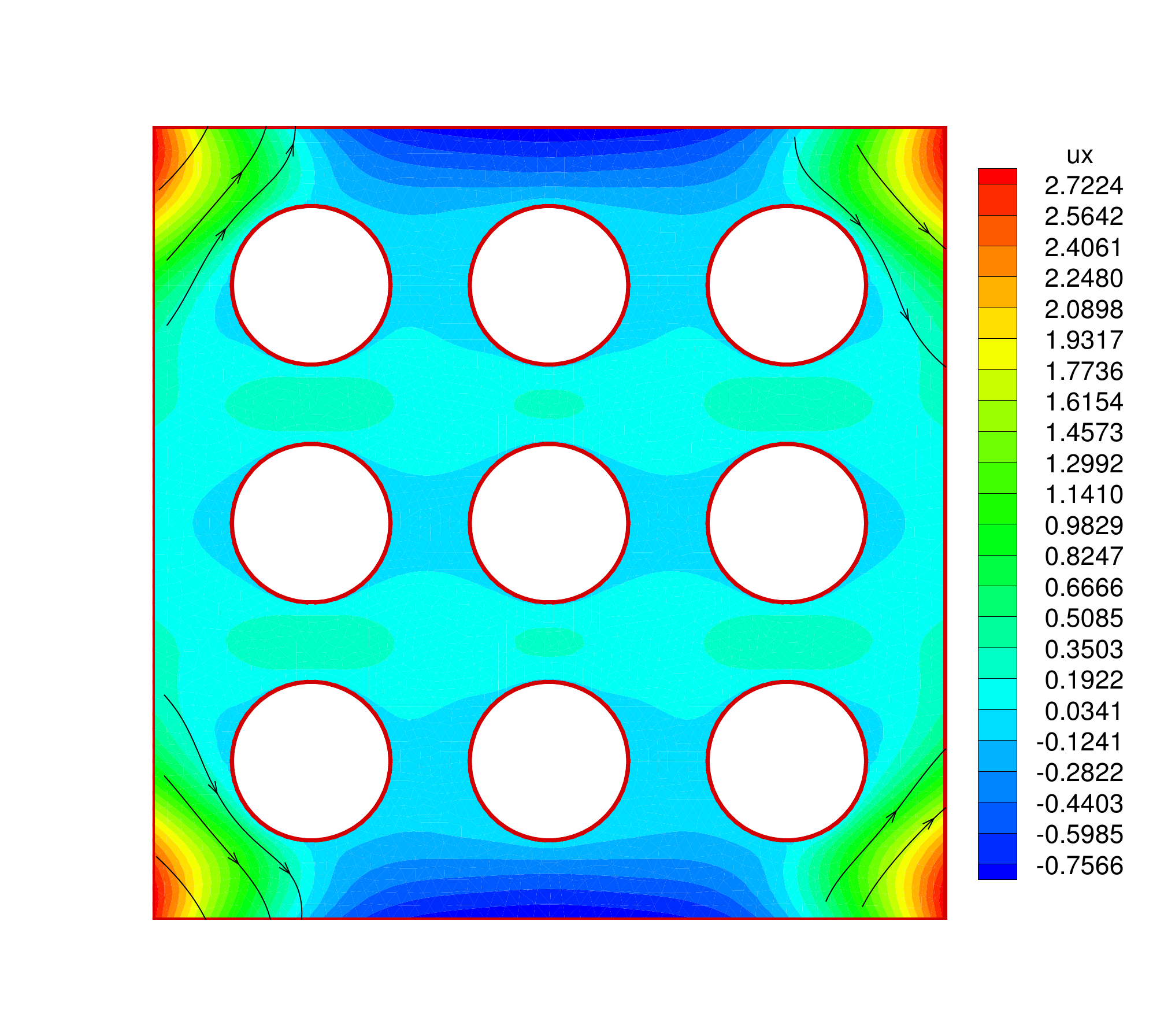}
    \hspace{-3mm}
    \includegraphics[width=0.4\textwidth]{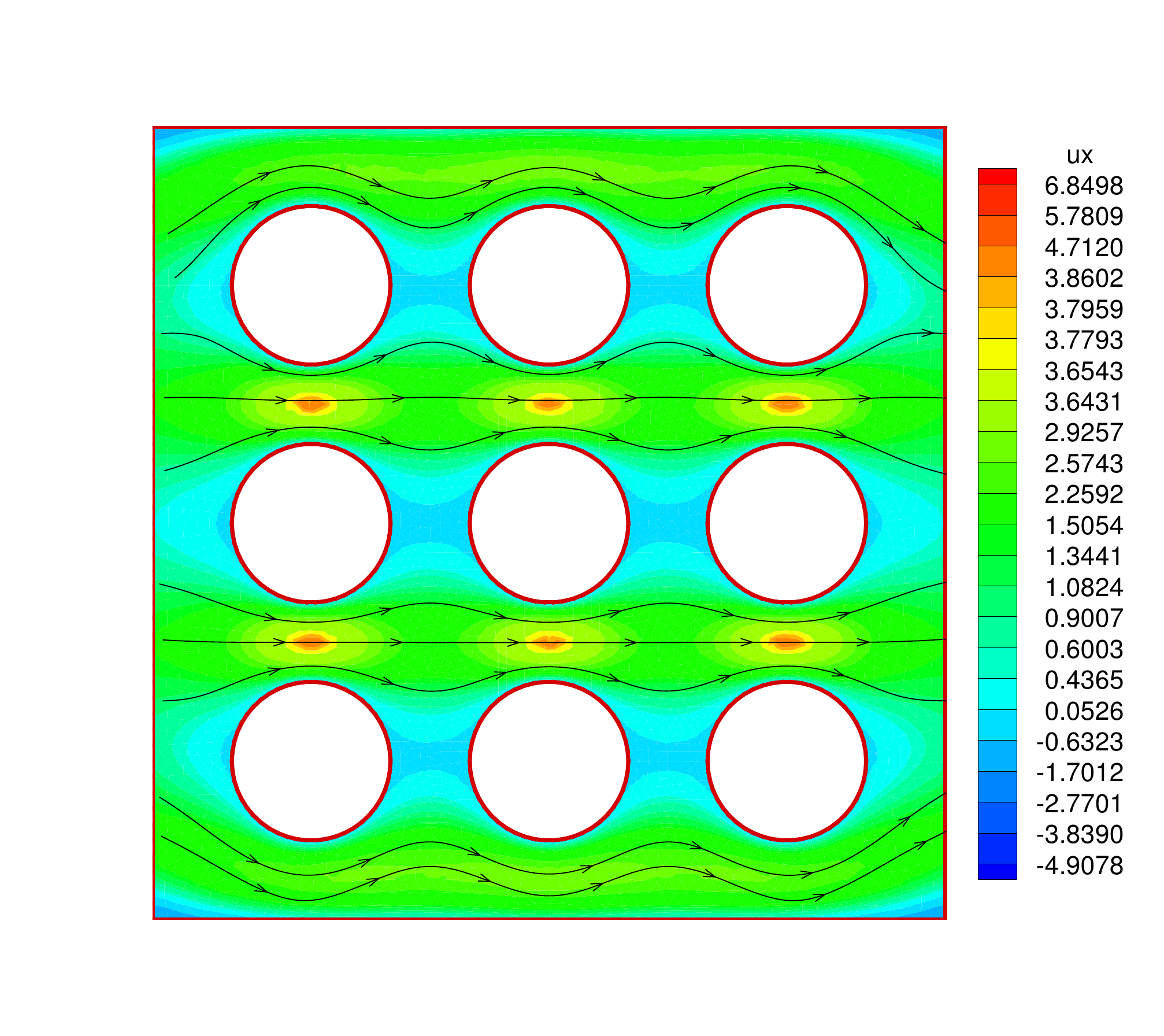}
    \caption{Basis function combination $\Phi_{LR}=\Phi_{L,1}+\Phi_{R,1}$.  Left: $\Phi_{LR}$ computed with CR$_2$ basis functions (\ref{CR2}); Right: $\Phi_{LR}$ computed with CR$_3$ basis functions  (\ref{CR3}). We show the stream-lines associated to these vector-valued functions. The colors represent their $x$-components.}
    \label{fig:base}
\end{figure}

\begin{remark}\label{rembub} 
Following \cite{MsFEMCR2}, one could think that the drawbacks of CR2 basis functions could be fixed if one added appropriate multi-scale bubble functions to the CR2 MsFEM space. One could thus introduce for any $T\in\TH$ the vector-valued velocity bubble $\Psi_{T,i}$ with associated pressure $\theta_{T,i}$, $i=1,2$ with $\Psi_{T,i}$, $\theta_{T,i}$ supported in $T$ and solution to 
\begin{align*}
-\Delta {\Psi}_{T,i}+\nabla \theta _{T,i} &=e_i& &\text{ on }\Omega
^{\varepsilon }\cap T, \\
\Div {\Psi}_{T,i} &=0 & &\text{ on }\Omega ^{\varepsilon }\cap T, \\
{\Psi}_{T,i} &=0 & &\text{ on }B^{\varepsilon }\cap T, \\
\nabla {\Psi}_{T,i}n-\theta _{T,i}n &= const & & \text{ on }F\cap \Omega
^{\varepsilon }\quad\text{ for all }F\in \mathcal{E}(T), \\
\int_{F}{\Psi}_{T,i} &=0 & &\text{ for all }F\in \mathcal{E}(T), \\
\int_{\Omega ^{\varepsilon }\cap T}\theta _{T,i} &=0. & &\quad
\end{align*}
We plot such a function at Fig. \ref{fig:bub} in a setting similar to that of Fig. \ref{fig:base} and observe that it could indeed restore the typical flow features lacking in the CR2 basis functions. However, CR2 MsFEM space, even enhanced with such bubble functions, would perform poorly with respect to the non-conformity error inherent to our method, cf. Remark \ref{remCR3trick}. This is why we have chosen not to consider the bubble functions in the present article, contrary to \cite{MsFEMCR2}.

\begin{figure}[h]
    \centering
    \includegraphics[width=0.4\textwidth]{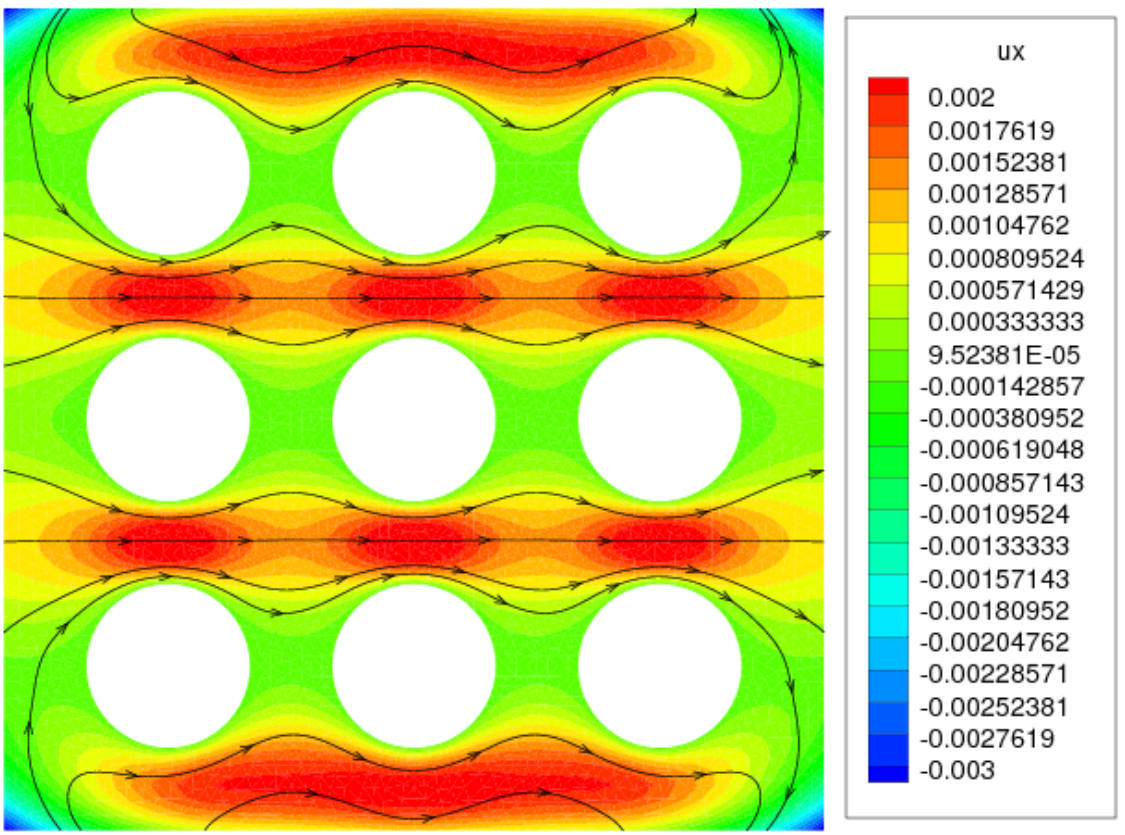}
    \caption{Bubble function $\Psi_{T,1}$ on the same mesh cell $T$ as at Fig. \ref{fig:base}. }
    \label{fig:bub}
\end{figure}
 
\end{remark}

\subsection{Periodic case}
The theoretical study of the MsFEM method introduced above will be performed only in the case of periodic perforations. Moreover, we shall need to be careful about the introduction of perforations near the boundary $\partial\Omega$. We adopt thus the following set of hypotheses. 

\begin{assumption}\label{AssPeriod}
$\Omega\in\mathbb{R}^2$ is a bounded simply connected polygonal domain, $
B^{\varepsilon} $ is a periodic set of holes inside $\Omega$, described below, and $
\Omega^\varepsilon=\Omega\setminus\bar{B^{\varepsilon}}$. 
Consider first the reference cell, the unit square $Y=(0,1)^2$, 
a domain $B\subset Y$ with sufficiently smooth boundary (the obstacle domain), and $\mathcal{F}=Y\setminus B$ 
(the fluid domain). Assume $dist(\partial B,\partial Y)>0$ and $\mathcal{F}$ connected. Take $
\varepsilon>0$ and define for any $i\in\mathbb{Z}^2$: $Y_{i}=
\varepsilon(Y+i)$, $B_{i}=\varepsilon(B+i)$, $\mathcal{F}_{i}=\varepsilon(\mathcal{F}+i)$. Finally, set 
\begin{equation*}
\mathcal{I}=\{i\in\mathbb{Z}^2:Y_{i}\subset\Omega\},\quad
B^{\varepsilon}=\cup_{i\in\mathcal{I}}B_{i} \text{ and }\Omega^{
\varepsilon}=\Omega\setminus\bar{B}^{\varepsilon}.
\end{equation*}
\end{assumption}
These definitions are illustrated in \figref{fig:domain setup}. Note that our definition of the perforated domain is slightly different from that of \cite{Mik96}, where~$\Omega^\varepsilon$ is perforated by all~$B_i$ that are enclosed in
$\Omega$. Here, we only leave the holes $B_i$ contained in a cell~$Y_i$ which
is itself inside~$\Omega$. 

\begin{figure}[h]
    \centering
    \def\svgwidth{0.6\textwidth}
    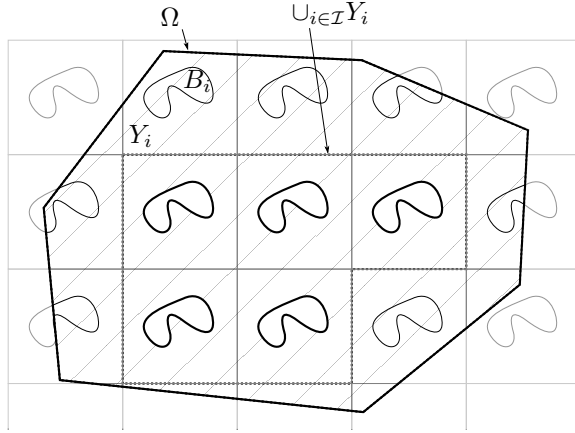
    \caption{Domain setup,~$\Omega^\varepsilon$ is crosshatched and
    its boundary is in bold lines}
    \label{fig:domain setup}
\end{figure}

We can now announce our main result, i.e. the error estimate for the CR$_3$ MsFEM method.

\begin{theorem}
\label{theo:main} Adopt Assumptions \ref{AssPeriod} on the perforated domain $\Omega
^{\varepsilon }$ and \ref{AssWeights2}--\ref{AssWeights3} about the mesh and the weighting functions used to set up the MsFEM method.
Assume moreover that the weighting functions are chosen as in (\ref{CR3}). Suppose also that ${f}$ and the
homogenized pressure $p^{\ast }$, cf. Section \ref{sec:homog}, are sufficiently smooth. 
The following error bound holds between the solution to the Stokes equations (\ref{genP}--\ref{genB}) and its MsFEM approximation \eqref{MsFEMu}--\eqref{MsFEMp} 
\begin{multline}\label{MainErrEst}
|{u}-{u}_{H}|_{H^{1}(\Omega)}+\varepsilon \Vert p-p_{H}\Vert_{L^2(\Omega)}
\\
\leq C\varepsilon \left( H+\sqrt{\varepsilon }+\sqrt{\frac{\varepsilon }{H}}\right)
(\Vert f\Vert_{\espHC}+\Vert p^{\ast }\Vert_{H^{2}(\Omega)})\,,
\end{multline}
where the constant $C$ depends only on the mesh regularity and the
perforation pattern $B$. 
\end{theorem}

The proof is postponed until Section \ref{sec:proof} and will use the results on homogenization from the next section and some technical lemmas from Section \ref{sec:preliminaries}.

\section{Homogenization for Stokes in two dimensions}
\label{sec:homog}

\subsection{The formal two-scale asymptotic expansion}
\label{sec:formal expansion}

We want to derive the asymptotic equation corresponding to~\eqref{genP},
in the limit $\varepsilon
\rightarrow0$. Let us do it first formally by introducing the two-scale
asymptotic expansions in terms of slow variable $x$ and the fast variable $
y=x/\varepsilon$. This procedure is quite well known, see for example \cite
{hornung,Sanchez80}. We describe it here for completeness and to set our
notations.
Let us expand~$u$ and~$p$ as
\begin{align*}
u(x) &= \sum_{k\ge 0} \varepsilon^k u_{k}(x,y)\,, & p(x) &= \sum_{k\ge 0} \varepsilon^k p_{k}(x,y)\,, &y=&\frac{x}{\varepsilon}
\end{align*}
where all the functions $u_{k},p_{k}$ are assumed $\mathbb{Z}^2$-periodic
in $y$, \emph{i.e.} 1-periodic with respect to both $y_{1}$ and $y_{2}$. We substitute these
series into the Stokes equations, use the chain rule, 
and get in the leading order $\frac{1}{\varepsilon^{2}}$
\begin{alignat*}{2}
-\Delta_{y}u_{0} &= 0&\quad&\text{ on }\mathcal{F}\,, \\
u_{0} &= 0&&\text{ on }\partial B\,.
\end{alignat*}
Reminding that $u_{0}$ is $\mathbb{Z}^2$-periodic in $y$, this gives $
u_{0}=0$ everywhere. At the next order, i.e. $\frac{1}{\varepsilon}$, we get
\begin{alignat*}{2}
-\Delta_{y}u_{1}+\nabla_{y}p_{0} &= 0&\quad&\text{ on }\mathcal{F}\,, \\
\Div_{y}u_{1} &= 0&&\text{ on }\mathcal{F}\,, \\
u_{1} &= 0&&\text{ on }\partial B\,.
\end{alignat*}
Reminding that $u_{1}$ and $p_{0}$ are $\mathbb{Z}^2$-periodic in $y$,
this gives $u_{1}=0$ everywhere and $p_{0}(x,y)=p^*(x)$.

At order $\varepsilon^{0}$ we get something less trivial
\begin{alignat*}{2}
-\Delta_{y}u_{2}+\nabla_{y}p_{1} &= f-\nabla_{x}p^*&\quad&\text{ on }
\mathcal{F}\,, \\
\Div_{y}u_{2} &= 0&&\text{ on }\mathcal{F}\,, \\
u_{2} &= 0&&\text{ on }\partial B\,,
\end{alignat*}
with the usual requirement that $u_{2}$ and $p_{1}$ be $\mathbb{Z}^2$-periodic
in $y$. This gives that $u_{2}$ and $p_{1}$ are the linear
combinations of the solutions to the following cell Stokes problems: for $
i=1,2$ find $w_{i}:\mathcal{F}\rightarrow\mathbb{R}^2$ and $
\pi_{i}:\mathcal{F}\rightarrow\mathbb{R}$,~$\mathbb Z^2$-periodic and solution of
\begin{alignat}{2}
-\Delta w_{i}+\nabla\pi_{i} &= e_{i}&\quad&\text{ on }\mathcal{F}\,,
\label{eq:corrector-lions}
\\
\Div w_{i} &= 0&&\text{ on }\mathcal{F}\,,
\notag
\\
w_{i} &= 0&&\text{ on }\partial B\,,
\notag
\\
\int_{\mathcal{F}}\pi_{i} &= 0\,. &&
\notag
\end{alignat}
We have thus, employing from now on the Einstein convention of summation on repeating indices
\begin{alignat}{2}
u_{2}(x,y) &= w_{i}(y)(f_{i}(x)-\partial_{i}p^*(x))\,, \label{u2}\\
p_{1}(x,y) &= \pi_{i}(y)(f_{i}(x)-\partial_{i}p^*(x))\,. \notag
\end{alignat}

The equation for $p^*(x)$ results from the next term in the asymptotic
expansion, at order~$\varepsilon$:
\begin{alignat*}{2}
-\Delta_{y}u_{3}+\nabla_{y}p_{2} &=
2(\nabla_{y}\cdot\nabla_{x})u_{2}-\nabla_{x}p_{1}&\quad&\text{ on }\mathcal{F}\,, \\
\Div_{y}u_{3} &= -\Div_{x}u_{2}&&\text{ on }\mathcal{F}\,, \\
u_{3} &= 0&&\text{ on }\partial B\,.
\end{alignat*}
plus periodicity conditions. The total outward flux of $u_3$ on~$\partial\mathcal F$ is zero in view of the boundary conditions,
so that the system of equations above has a solution if and only if $\average{\Div
_{x}u_{2}}=0 $ where $\average{\cdot}$ stands for the average over $\mathcal{F}$:
\begin{equation}\label{devavg}
\average{v}=\frac{1}{|\mathcal{F}|}\int_{\mathcal{F}}v(y)\, dy.
\end{equation}
This gives Darcy equation for $p^*$:
$$
    \Div(\average{w_{i}}(f_{i}-\partial_{i}p^*)) = 0\quad\text{ on }\Omega\,.
$$
We see that $u_{3}$ and $p_{2}$ are the linear combinations of the solutions
to yet another cell Stokes problem: for $i,j=1,2$ find $\gamma_{ij}:\mathcal{F}\rightarrow\mathbb{R}^2$ and $\vartheta_{ij}:\mathcal{F}
\rightarrow\mathbb{R}$,~$\mathbb Z^2$-periodic and solution of
\begin{alignat}{2}
-\Delta\gamma{}_{ij}+\nabla\vartheta{}_{ij} &=
2\partial_{j}w_{i}-\pi_{i}e_{j}&\quad&\text{ on }\mathcal{F}\,, \notag\\
\Div\gamma_{ij} &= -w_{i}\cdot e_{j}+\average{w_{i}\cdot e_{j}}&&\text{ on }\mathcal{F}\,, \label{gammaij}\\
\gamma_{ij} &= 0&&\text{ on }\partial B\,, \notag\\
\int_{\mathcal{F}}\vartheta_{ij} &= 0\,. \notag
\end{alignat}
We have thus
\begin{align}\label{u3}
u_{3}(x,y) &= \gamma_{ij}(y)\partial_{j}(f_{i}(x)-\partial_{i}p^*(x))\,,
\\
p_{2}(x,y) &=
\vartheta_{ij}(y)\partial_{j}(f_{i}(x)-\partial_{i}p^*(x))\,.
\notag
\end{align}

From now on, we will denote by~$u^*$ the homogenized velocity, i.e. the first non-zero terms in the expansion of~$u$:
\begin{equation}
u^* = \varepsilon^2 u_2(x,y) = \varepsilon^2 (w_i)_\varepsilon\left( f_i - \partial_i p^* \right)\,.
\label{eq:homogenized velocity}
\end{equation}
\textbf{Notation:} We use a shorthand  $(\cdot)_\eps$ to indicate the rescaling by $\eps$. Thus, $(\phi)_\eps(x)=\phi\left(\frac{x}{\eps}\right)$ for any $\Z^2$ periodic function $\phi$. 

The procedure above does not provide boundary conditions for $p^*$.
The good choice for these is to ensure that the normal component of the averaged homogenized velocity
vanishes on the boundary, i.e. $n\cdot \average{u^*}=0$ on $\partial\Omega$.

\subsection{A rigorous homogenization estimate}

The homogenization of the Stokes equations was first rigorously
studied in \cite{Tartar80}, where the weak $L^2$ convergence for the velocity
and the strong $L^2$ convergence for the pressure were established.
The strong $L^2$ convergence for the velocity was later proven in
\cite{allaire89}. However, for our purposes it is desirable to have a convergence result in $H^{1}$ and, moreover, an estimate of the
homogenization error in this norm. Such an estimate is available in
\cite{Mik96} with a relative error of order $\sqrt[6]{\varepsilon}$.
We shall improve it here to $\sqrt{\varepsilon}$ and provide another approach to the proof 
(as already noted, our definition of the perforated domain is slightly different from that in \cite{Mik96}).
Our homogenization result is as follows.

\begin{theorem}
\label{HomMain}Recall Assumption \ref{AssPeriod} and let $u,p$ be the solution to the Stokes equations (\ref{genP}
)--\eqref{genB}, $p^{\ast}$ be the solution to the Darcy equation
\begin{alignat}{2}
    \Div(\average{w_{i}}(f_{i}-\partial_{i}p^*)) &= 0&\quad&\text{ on }\Omega\,,
\label{darcy}
\\
n\cdot\average{w_{i}}(f_{i}-\partial_{i}p^*) &= 0
&&\text{ on } \partial\Omega\,,
\label{DarcyB}
\end{alignat}
and $u^\ast$ be defined by \eqref{eq:homogenized velocity} with $w_i$ extended by 0 inside $B$.
Assuming that $f$ and $p^{\ast}$ are sufficiently smooth there holds
\begin{align}
\Vert p-p^{\ast}\Vert_{L^2(\Omega^\varepsilon)} & \leq C\varepsilon^{\frac{1}{2}}\Vert f-\nabla
p^{\ast}\Vert_{H^{2}(\Omega) \cap C^{1}(\bar\Omega)} \,,  \label{errP} \\
|u-u^{\ast}|_{H^{1}(\Omega^\varepsilon)} & \leq C\varepsilon^{\frac{3}{2}}\Vert f-\nabla
p^{\ast}\Vert_{H^{2}(\Omega) \cap C^{1}(\bar\Omega)}  \label{errH1} \,,\\
\Vert u-u^{\ast}\Vert_{L^2(\Omega^\varepsilon)} & \leq C\varepsilon^{\frac{5}{2}
}\Vert f-\nabla p^{\ast}\Vert_{H^{2}(\Omega) \cap C^{1}(\bar\Omega)}  \label{errL2}
\end{align}
where~$C$ is independent of $\varepsilon$.
\end{theorem}

\begin{remark}
The estimate for the velocity in the $H^{1}$ norm essentially says that the
relative error is of order $\sqrt{\varepsilon}$. Indeed, the velocity itself
is of order $\varepsilon^{2}$, but its derivatives are of order $\varepsilon$
since both the exact solution and its homogenized approximation oscillate on
the length scale $\varepsilon$. Also note that the deterioration of order
$\sqrt{\varepsilon}$ is due to the boundary layers near $\partial\Omega$.
Indeed, $ u^{\ast}$ does not satisfy the boundary condition $u=0$ on $
\partial\Omega$, which worsens the approximation near the boundary.
Technically, this is taken into account by the introduction of the cut-off
function $\eta^{\varepsilon}$ in the forthcoming proof. If the boundary
layers were absent, which would be the case, for example, under the periodic
boundary conditions over a rectangular box $\Omega=(0,\varepsilon
n)\times(0,\varepsilon m)$ with $n,m\in\mathbb{N}$, the a priori error
estimate would give the relative error of order $\varepsilon$. Indeed, inspecting the forthcoming proof, one can see 
that neither Lemma \ref{infsupb} nor the cut-off functions $\eta^\eps$ are no longer needed in this case and the final result becomes
\begin{equation*}
\|u-u^{\ast}\|_{L^2(\Omega)}+\eps|u-u^{\ast}|_{H^{1}(\Omega)}+\varepsilon^2\Vert p-p^{\ast}\Vert_{L^2(\Omega)}\leq
C\varepsilon^{3}\Vert f-\nabla p^{\ast}\Vert_{H^{2}(\Omega)}\,.
\end{equation*}
\end{remark}

Before providing the proof of Theorem~\ref{HomMain}, let us establish two technical
lemmas of the inf-sup type related to the divergence free constraint (Lemma \ref{infsupq}) and to
the boundary conditions for the velocity (Lemma \ref{infsupb}). All these results are proved under Assumption \ref{AssPeriod}.

\begin{lemma}
\label{infsupq} For any $q\in L^2_0(\Omega^{\varepsilon })$ there exists $v\in H_{0}^{1}(\Omega
^{\varepsilon })^{2}$ such that
\begin{equation}
    \Div v=q\text{ on }\Omega^\varepsilon\text{ and }|v|_{H^{1}(\Omega
^{\varepsilon })}\leq \frac{C}{\varepsilon }\Vert q\Vert_{L^2(\Omega
^{\varepsilon })}\,,  \label{infsup1}
\end{equation}
where $C>0$ is a constant independent of $\varepsilon $.
\end{lemma}

\begin{proof}
Let us take any $q\in L^2_0(\Omega)$ such that $q=0$ on $B^\varepsilon$. Using
\cite[Corollary 2.4, p.24]{GiraultRaviart}, we can pick some $w\in H_{0}^{1}(\Omega)^{2}$
such that
\begin{equation}
\Div w=q\text{ on }\Omega\text{ and }|w|_{H^{1}(\Omega)}\leq
C\Vert q\Vert_{L^2(\Omega)}.  \label{vglob}
\end{equation}
This gives us a velocity field~$w$ on~$\Omega$ that does not satisfy the boundary
conditions on~$\Omega^\varepsilon$, \emph{i.e.}~$w\neq 0$ on~$\partial
B^\varepsilon$. Using it as a starting point, we can construct
an admissible velocity field on each cell~$\mathcal{F}_i$, proceeding cell by cell, as follows.

Let us pick any $i\in\mathcal{I}$, denote by $w_{Y_i}$, $q_{Y_i}$ the restrictions of $w$, $q$
to the cell $Y_i$ and map them to the reference cell $Y$:
$$
   \hat w_Y(x) = w_{Y_i}(\varepsilon (x+i)) ,\quad
    \hat q_Y(x) = \varepsilon q_{Y_i}(\varepsilon (x+i)).
$$
The scalings are chosen so that 
$$
\Div\hat{w}_{Y}=\hat{q}_{Y}\text{ on }Y.
$$
A standard trace theorem assures that there exists $r\in H^1(\mathcal{F})^2$ such that
$$
r=\hat{w}_{Y}\text{ on }\partial Y,\quad r=0\text{ on }\partial B \quad\text{and}\quad 
\|r\|_{H^1(\mathcal{F})}\le C\|\hat{w}_{Y}\|_{H^{\frac 12}(\partial Y)} \le C\|\hat{w}_{Y}\|_{H^1(Y)}.
$$
Using again the corollary from \cite{GiraultRaviart} mentioned above and noting that 
$$
\int_\mathcal{F}\left(\hat{q}_{Y}-\Div r\right)
=
\int_\mathcal{F}\hat{q}_{Y}-\int_{\partial \mathcal{F}}r\cdot n
=
\int_Y\hat{q}_{Y}-\int_{\partial Y}\hat{w}_{Y}\cdot n
=
\int_Y(\hat{q}_{Y}-\Div\hat{w}_{Y})
=0
$$
we can construct $z\in H^1_0(\mathcal{F})^2$ with 
$$
\Div z=\hat{q}_{Y}-\Div r
\text{ and }
\|z\|_{H^1(\mathcal{F})} 
\le C\|\hat{q}_{Y}-\Div r\|_{L^2(\mathcal{F})}
\le C\|\hat{w}_{Y}\|_{H^1(Y)}.
$$
Setting now 
$\hat{v}_{Y}\in(H^{1}(\mathcal{F}))^{2}$ as $\hat{v}_{Y}=r+z$ we observe
$$
\hat{v}_{Y}=\hat{w}_{Y}\text{ on }\partial Y,\quad \hat{v}_{Y}=0\text{ on }\partial B,\quad 
\Div\hat{v}_{Y}=\hat{q}_{Y}\quad\text{and}\quad 
\|\hat{v}_{Y}\|_{H^1(\mathcal{F})} \le C\|\hat{w}_{Y}\|_{H^1(Y)}.
$$
Note that the constants $C$ in the above bounds depend only on the geometry of $\mathcal{F}$. In particular,
they are obviously $\eps$-independent.
We now rescale the cell $Y$ back to the cell $Y_{i}$ of
size $\varepsilon$ and define  $v_{Y_i}\in H^1(F_i)$ by $\hat v_Y(x) = v_{Y_i}(\varepsilon (x+i))$. 
Recalling the scalings of the functions and of their norms
\[
\left\{
\begin{aligned}
    \hat v_Y(x) &= v_{Y_i}(\varepsilon (x+i))\\
    \hat w_Y(x) &= w_{Y_i}(\varepsilon (x+i)) \\
    \hat q_Y(x) &= \varepsilon q_{Y_i}(\varepsilon (x+i))
\end{aligned}
\right.
\ \Longrightarrow\ 
\left\{
\begin{aligned}
    \left| \hat v_Y \right|_{H^1} &= \left| v_{Y_i} \right|_{H^1}\\
    \Vert \hat w_Y \Vert_{H^1} &= \left( |w_{Y_i}|_{H^1}^2 + \frac 1{\varepsilon^2} \Vert w_{Y_i} \Vert_{L^2} \right)^{\frac 12} \le \frac 1\varepsilon\Vert w_{Y_i} \Vert_{H^1}\\
    \Vert \hat q_Y \Vert_{L^2} &= \Vert q_{Y_i} \Vert_{L^2}
\end{aligned}
\right.
\]
we conclude
\begin{equation}
\Div v_{Y_{i}}=q_{Y_{i}}\text{ on }\mathcal{F}_i ,~v_{Y_{i}}=w_{Y_{i}}
\text{ on }\partial Y_{i},\text{ }v_{Y_{i}}=0\text{ on }\partial B_{i}\,,
\label{vloc}
\end{equation}
and
\begin{equation*}
|v_{Y_{i}}|_{H^{1}(\mathcal{F}_i )}\leq
\frac{C}{\varepsilon}\Vert w_{Y_{i}}\Vert_{H^{1}(Y_{i})}\,.
\end{equation*}

We now collect all the pieces $v_{Y_{i}}$ into $v\in(H_{0}^{1}(\Omega^\varepsilon))^{2}$ such
that $v|_{Y_{i}}=v_{Y_{i}}$ for any cell $Y_{i}$, $i\in\mathcal{I}$ and let $v=w$
on $\Omega_{b}:=\Omega\setminus\cup_{i\in\mathcal{I}}Y_{i}$. Such a function $v$ meets
all the requirements of the lemma. Indeed, $\Div v=q$ on $\Omega^\varepsilon$
and
\begin{align*}
\text{ }|v|_{H^{1}(\Omega^\varepsilon)}^{2} & =
|w|_{H^{1}(\Omega_{b})}^{2}+\sum_{i\in\mathcal{I}}|v_{Y_{i}}|_{H^{1}(\mathcal{F}_i )}^{2}\\
&\leq|w|_{H^{1}(\Omega_{b})}^{2} +\sum_{i\in \mathcal I}\frac{C}{\varepsilon^{2}}\Vert w\Vert_{H^{1}(Y_{i})}^{2}
\leq\frac{C}{\varepsilon^{2}}\Vert w\Vert_{H^{1}(\Omega)}^{2}
\leq\frac{C}{\varepsilon^{2}}\Vert q\Vert_{L^2(\Omega^\varepsilon)}^{2}\,.
\end{align*}
\end{proof}

Lemma~\ref{infsupq} is very close  to the results on the
restriction operator in \cite{Tartar80,jager95}. Our next result is essentially taken from 
\cite{Mik96} but we provide here a slightly simpler construction that suits well to polygonal domains.

\begin{lemma}
\label{infsupb} For any $g\in (C^{1}(\bar\Omega))^2$ with
$\Div g = 0$ on~$\Omega$, $g\cdot n=0$ on $
\partial\Omega$, $\delta>0$ small enough,
there exists $v\in(H^{1}(\Omega))^{2}$ such that $\supp\,v\subset O^{\delta}:=
\{x\in\Omega:dist(x,\partial\Omega)<\delta\}$ and
\begin{equation*}
v=g\mbox{ on }\partial\Omega,\quad\Div v=0\mbox{ on }\Omega,\mbox{ and }
|v|_{H^{1}(\Omega)}\le\frac{C}{\sqrt{\delta}}\Vert g\Vert_{C^{1}(\Omega)}
\end{equation*}
where $C>0$ is a constant independent of $\delta$.
\end{lemma}

\begin{proof} According to \cite[Theorem 3.1]{GiraultRaviart}, we can write~$g = \nabla^\perp \Psi$
for some $\Psi \in H^1(\Omega)$,
with $\nabla^{\perp}=(-\partial_{2},\partial_{1})^{T}$.
In fact, $\Psi \in C^2(\bar\Omega)$ as seen from the explicit construction
$$
\Psi(x)=-\int_a^x g^\perp\cdot d\vec{R}
$$
where $x$ is any point in $\Omega$,  $a$ is a fixed point in $\Omega$, $g^\perp=(-g_2,g_1)^T$ and
the integral is taken over any curve connecting $a$ and $x$ (parameterized by a vector function $\vec{R}$).
Note that~$g\cdot n = \nabla \Psi \cdot \tau = 0$ on $\partial\Omega$, where~$n$ (resp.~$\tau$) is
the unit vector normal (resp. tangent) to~$\partial\Omega$, so we can choose~$\Psi$
such that~$\Psi(x) = 0$ on~$\partial\Omega$.
We can now pick a cut-off function~$\eta \in C^\infty(\Omega)$ such that~$\eta(x) = 1$ on~$O^{\delta/2}$,
$\eta(x) = 0$ on~$\Omega \setminus O^\delta$, 
and $\Vert \nabla \eta \Vert_{L^\infty} \le \frac{C}{\delta}$,
$\Vert \nabla^2 \eta \Vert_{L^\infty} \le \frac{C}{\delta^2}$. Here and below,
$C$ stands for positive constants independent of~$\delta$.
Now, setting~$v = \nabla^\perp (\eta \Psi)$, we have~$\Div v = 0$, $\supp\,v\subset O^{\delta}$, $v=g$ on~$\partial\Omega$, 
and
  \begin{align*}
    |v|_{H^{1}({\Omega})}  &
    =\|{\nabla}{\nabla}^{{\perp}}({\eta}{\Psi})\|_{L^2({O}^{{\delta}})}
    \\
    &{\leq}\|{\eta}{\nabla}{\nabla}^{{\perp}}{\Psi}\|_{L^2({O}^{{\delta}})}+2\|({\nabla}{\eta})({\nabla}^{{\perp}}{\Psi})\|_{L^2({O}^{{\delta}})}+\|{\Psi}{\nabla}{\nabla}^{{\perp}}{\eta}\|_{L^2({O}^{{\delta}})}\\
    &
    {\le}C\|{\nabla}g\|_{L^2({O}^{{\delta}})}+{\frac{C}{{\delta}}}\|g\|_{L^2({O}^{{\delta}})}+{\frac{C}{{\delta}^{2}}}\|{\Psi}\|_{L^2({O}^{{\delta}})}\\
    &
    {\le}C{\sqrt{{\delta}}}\|{\nabla}g\|_{{L^\infty} ({\Omega} )}+{\frac{C}{{\sqrt{{\delta}}}}}\|g\|_{{L^\infty}({\Omega} )}+{\frac{C}{{\delta}^{3/2}}}\|{\Psi}\|_{{L^\infty}({O}^{{\delta}} )}
  \end{align*}
  since meas$( O^{\delta}) \le C\delta$. We observe now that any
  point $x \in O^{\delta}$ can be connected to a point $y \in \partial
  \Omega$ by a segment of length no greater than $\delta$ lying in $O^\delta$.
  Reminding that $\Psi ( y) = 0$ \ and using the Taylor expansion of order 0
  gives $| \Psi ( x) | \leq \delta | \nabla \Psi ( z) |$ for some point
  $z$ lying on this segment. Thus,
  \[ \| \Psi \|_{{L^\infty} ( O^{\delta})} \leq \delta \| \nabla \Psi \|_{{L^\infty} (
     O^{\delta})} = \delta \| g \|_{{L^\infty} ( \Omega )} \]
  which yields the result.
\end{proof}

We remind also a Poincar\'{e} inequality on the perforated domain.
\begin{lemma}Under Assumption \ref{AssPeriod}, for any $\phi \in
H_{0}^{1}(\Omega^{\varepsilon })$
\begin{equation}
\Vert\phi \Vert_{L^2(\Omega^{\varepsilon })}\leq C\varepsilon |\phi
|_{H^{1}(\Omega^{\varepsilon })}.  \label{Poin}
\end{equation}
 with a constant $C>0$ independent of $\varepsilon$. 
\end{lemma}
\begin{proof}
This is a corollary of Lemma \ref{lem:poincare perforated broken} proven below. The present lemma can be also proven directly, cf. for example \cite{hornung} or \cite[Appendix A.1]{MsFEMCR2}. The definition of the perforated domain in these references is slightly different from the present article (the perforations are maintained near the boundary) but this does not change essentially the proof, since the band where the perforations are eliminated is of width $\sim\eps$.\qquad
\end{proof}

\textit{Proof of Theorem \ref{HomMain}.} Consider 
\[
    w^{\prime }_{i}=w_{i}-\left|\mathcal{F}\right|\average{w_{i}}
      =w_{i}-\int_Y w_{i}
\]
with $w_{i}$ extended by 0 inside $B$ and observe that $\Div w^{\prime }_i=0$ on $Y$ (in the sense of distributions), $w^{\prime }_i$ is $\mathbb{Z}^2$-periodic and of zero mean over $Y$. Thus (cf.~\cite[p. 6]{Jikov1994}) there exists a $
\mathbb{Z}^2$-periodic function $\psi_{i}$ such that
\begin{equation*}
w_{i}-|\mathcal{F}|\average{w_{i}}=\nabla^{\perp}\psi_{i}\mbox{ on }Y\,.
\end{equation*}
In fact, $\psi_{i}$ can be assumed as smooth on $\mathcal{F}$ as we want, as seen from its explicit construction 
\begin{equation*}
\psi_{i}(x)=\int_{0}^{1}\left(x_{2}[w^{\prime }_{i}]_{1}(tx)-x_{1}[w^{\prime
}_{i}]_{2}(tx)\right)\, dt
\end{equation*}
and the fact that $w_i$ is smooth thanks to our assumptions on perforation $B$.

Assumptions \ref{AssPeriod} also implies that there exists a constant~$c>0$ such that~$O^{\delta}$ with $\delta=c\eps$ does not intersect
the holes~$\cup_{i\in \mathcal I} B_i$ (here, $O^{\delta}$ stands for the band of width $\delta$ near $\partial\Omega$ as in Lemma 
\ref{infsupb}).
Let us choose a cut-off function $\eta^{\varepsilon}\in
C^{\infty}(\bar{\Omega})$ with $\eta^{\varepsilon}=\frac{
\partial\eta^{\varepsilon}}{\partial n}=0$ on $\partial\Omega$, $
\eta^{\varepsilon}(x)=1$ on $\Omega\setminus O^{\delta}$ and 
\begin{equation}\label{etabnd}
\Vert\eta^{\varepsilon}\Vert_{L^{\infty}(\Omega)}=1,\quad\Vert1-\eta^{
\varepsilon}\Vert_{L^2(\Omega)}\leq C\sqrt{\varepsilon},\quad|\eta^{
\varepsilon}|_{H^{1}(\Omega)}\leq\frac{C}{\sqrt{\varepsilon}}
,\quad|\eta^{\varepsilon}|_{H^{2}(\Omega)}\leq\frac{C}{\varepsilon^{3/2}}\,.
\end{equation}

We now consider the expansion of the velocity  of order 3 in~$\varepsilon$  and correct it using the cut-off $\eta^\eps$ 
to take into account the boundary layer:
\begin{equation}
    \label{eq:u star three}
\begin{aligned}
u^{\varepsilon,3}&=
\varepsilon^{2}|\mathcal{F}|\average{w_{i}}(f_{i}-\partial_{i}p^{\ast})
+\varepsilon^{3}\nabla^{\perp}((
\psi_{i})_{\varepsilon}\eta^{\varepsilon})(f_{i}-\partial_{i}p^{\ast})
\\
&\quad+ \varepsilon^{3}(\gamma_{ij})_{\varepsilon}\eta^{\varepsilon}
\partial_{j}(f_{i}-\partial_{i}p^{\ast})
\end{aligned}
\end{equation}
We assume here that both $v_i$ and $\gamma_{ij}$ are extended by 0 inside $B$ so that $u^{\varepsilon,3}$ is well defined 
on the whole of $\Omega$. 
Remind that $\eta^{\varepsilon}=1$ on $\Omega\setminus O^{\delta}$ so that
the expression for $u^{\varepsilon,3}$ simplifies on this portion of $\Omega$ to
\begin{equation}  \label{ueps3s}
u^{\varepsilon,3}
=
u^\ast + \varepsilon^3 u_3
=
\varepsilon^{2}(w_{i})_{\varepsilon}(f_{i}-
\partial_{i}p^{\ast})+\varepsilon^{3}(\gamma_{ij})_{\varepsilon}
\partial_{j}(f_{i}-\partial_{i}p^{\ast}).
\end{equation}
It means in particular that $u^{\varepsilon,3}$ vanishes on the holes $B_i$, $i\in\mathcal{I}$ which are all inside $\Omega\setminus O^{\delta}$.
Let us compute $\Div u^{\varepsilon,3}$ knowing that the divergence of the first term in \eqref{eq:u star three} vanishes by \eqref{darcy}:
\begin{align*}
\Div u^{\varepsilon,3} &=
 \varepsilon^{2}(\nabla^{\perp}\psi_{i})_{\varepsilon}
\eta^{\varepsilon}\cdot
\nabla(f_{i}-\partial_{i}p^{\ast})+\varepsilon^{3}(\psi_{i})_{\varepsilon}
\nabla^{\perp}\eta^{\varepsilon}\cdot\nabla(f_{i}-\partial_{i}p^{\ast}) \\
&\quad + \varepsilon^{2}(\Div\gamma_{ij})_{\varepsilon}\eta^{\varepsilon}
\partial_{j}(f_{i}-\partial_{i}p^{\ast})
\\
&\quad+\varepsilon^{3}(\gamma_{ij})_{
\varepsilon}\cdot(\nabla\eta^{\varepsilon})\partial_{j}(f_{i}-
\partial_{i}p^{\ast})+\varepsilon^{3}(\gamma_{ij})_{\varepsilon}\eta^{
\varepsilon}\cdot\nabla\partial_{j}(f_{i}-\partial_{i}p^{\ast}) 
\end{align*}
Grouping together the terms of order $\eps^2$, using equation (\ref{gammaij}) for $\Div\gamma_{ij}$, and denoting by ${G}_{\varepsilon}$ all the terms of order $\eps^3$
\begin{align*}
{G}_{\varepsilon} &:=\varepsilon^{3}(\psi_{i})_{\varepsilon}\nabla^{\perp}
\eta^{\varepsilon}\cdot\nabla(f_{i}-\partial_{i}p^{\ast})
\\
&\qquad
+\varepsilon^{3}(
\gamma_{ij})_{\varepsilon}\cdot(\nabla\eta^{\varepsilon})\partial_{j}(f_{i}-
\partial_{i}p^{\ast})+\varepsilon^{3}(\gamma_{ij})_{\varepsilon}\eta^{
\varepsilon}\cdot\nabla\partial_{j}(f_{i}-\partial_{i}p^{\ast}) 
\end{align*}
we proceed with the calculation as
\begin{align*}
\Div u^{\varepsilon,3} &= \varepsilon^{2}\eta^{\varepsilon}(w_{i}-|\mathcal{F}|\average{w_{i}}
-(w_{i}-\average{w_{i}}))_{\varepsilon}\cdot\nabla(f_{i}-\partial_{i}p^{
\ast})+{G}_{\varepsilon} \\
&= \varepsilon^{2}\eta^{\varepsilon}|B|\Div(\average{w_{i}}(f_{i}-
\partial_{i}p^{\ast}))+{G}_{\varepsilon} 
={G}_{\varepsilon} \,.
\end{align*}
Note that this equality also holds trivially inside any hole $B_k$, $k\in\Z^2$ since both sides vanish there. 
Thanks to the bounds (\ref{etabnd}), we conclude
\[
\Vert{G}_{\varepsilon}\Vert_{L^2(
\Omega^\varepsilon)} \le C \varepsilon^{\frac{5}{2}}\Vert f-\nabla
p^{\ast}\Vert_{H^{2}(\Omega) \cap C^{1}(\bar\Omega)}\,,
\]
with~$C>0$ independent of~$\varepsilon$.
We also note for future use 
\begin{equation}\label{defg}
u-u^{\varepsilon,3}=g:=-\varepsilon^{2}|\mathcal{F}|\average{w_{i}}(f_{i}-\partial_{i}p^{\ast})
\text{ on }\partial\Omega\,.
\end{equation} 

We now turn to estimates for the residual in \eqref{genP} caused by the
homogenization. One of the technical difficulties consists in the presence of ``virtual'' holes $B_i$ near $\partial\Omega$ that are in fact in the fluid domain $\Omega^\eps$ according to our conventions, cf. Assumption \ref{AssPeriod} and Fig. \ref{fig:domain setup} (the gray hole contours in the periodic cells cut by the boundary $\partial\Omega$). One should thus define properly the cell velocities $w_i$ inside $B_k$. The usual extension by 0, which worked fine in all the previous calculations, does not suffice here because it does not give a twice differentiable function. We thus introduce an extension $\tilde{w}_i$ of $w_i$ from $\mathcal{F}$ to $Y$ such that $\tilde{w}_i=w_i$ on $\mathcal{F}$ and $\tilde{w}_i$ is of class $C^2$ on  $Y$. Now, consider
$$
\tilde{u}^*  = \varepsilon^2 (\tilde w_i)_\varepsilon\left( f_i - \partial_i p^* \right)\,.
$$
Similarly, let $\tilde{\pi}_i$ be an extension of $\pi_i$ from $\mathcal{F}$ to $Y$ such that $\tilde{\pi}_i=\pi_i$ on $\mathcal{F}$ and $\tilde{\pi}_i$ is of class $C^1$ on  $Y$. Introduce the expansion of first order in $\eps$ for the pressure
\begin{equation}
    \label{eq:definition p epsilon one}
\tilde p^{\varepsilon ,1}
=
p^{\ast }+\varepsilon (\tilde\pi_{i})_{\varepsilon
}(f_{i}-\partial_{i}p^{\ast })\,.
\end{equation}
Thus, the residual due to the homogenization in eq. (\ref{genP}) is given everywhere on $\Omega^\eps$ by
\begin{equation}
    \label{eq:stokes expansion residual}
\begin{aligned}
F_{\varepsilon }&:=-\Delta (u-\tilde u^{\ast })+\nabla (p-\tilde p^{\varepsilon ,1})
\\
&=f+(\Delta \tilde w_{i}-\nabla \tilde\pi_{i})_{\varepsilon }(f_{i}-\partial_{i}p^{\ast})
\\
&\quad+2\varepsilon (\nabla \tilde w_{i})_{\varepsilon }\nabla (f_{i}-\partial
_{i}p^{\ast })
+\varepsilon ^{2}(\tilde w_{i})_{\varepsilon }\Delta (f_{i}-\partial_{i}p^{\ast}) \\
&\quad-\nabla p^{\ast }-\varepsilon (\tilde\pi_{i})_{\varepsilon }\nabla
(f_{i}-\partial_{i}p^{\ast }). \\
&
\end{aligned}
\end{equation}
Rearranging the terms yields
\begin{align*}
F_{\varepsilon }&=2\varepsilon (\nabla \tilde w_{i})_{\varepsilon }\nabla
(f_{i}-\partial_{i}p^{\ast })+\varepsilon ^{2}(\tilde w_{i})_{\varepsilon }\Delta
(f_{i}-\partial_{i}p^{\ast })-\varepsilon (\tilde\pi_{i})_{\varepsilon }\nabla
(f_{i}-\partial_{i}p^{\ast })\\
&+(\Delta \tilde w_{i}-\nabla \tilde\pi_{i}+e_i)_{\varepsilon }(f_{i}-\partial_{i}p^{\ast})
\end{align*}
The terms in the first line above are of order $\eps$ or higher. The terms in the second line are of order 1, but they vanish in fact at all the fluid cells $\mathcal{F}_i$, $i\in\mathcal{I}$. Since the measure of the remaining part $\Omega^\eps\setminus\cup_{i\in\mathcal{I}}\mathcal{F}_i$ is of order $\eps$, we get
\begin{equation}
\Vert F_{\varepsilon }\Vert_{L^2(\Omega^{\varepsilon })}\le C
\sqrt{\varepsilon} \Vert f-\nabla p^{\ast }\Vert_{\espHC}.
\label{eq:homogenization partial residual}
\end{equation}

We summarize all the derived bounds as follows: the functions $u-\tilde u^{\ast
},u-u^{\varepsilon,3}\in H^{1}(\Omega^{\varepsilon })^{2}$ and $
p-\tilde p^{\varepsilon ,1}\in L^2(\Omega^{\varepsilon })$ satisfy
\begin{alignat}{2}
-\Delta (u-\tilde u^{\ast})+\nabla (p-\tilde p^{\varepsilon ,1}) &=F_{\varepsilon }
&\quad&\text{ on }\Omega^{\varepsilon }
\label{StErr} \\
\Div(u-u^{\varepsilon,3}) &={G}_{\varepsilon } &&\text{ on }\Omega
_{\varepsilon }
\notag \\
u-u^{\varepsilon,3} &=0 && \text{ on }\partial B^{\varepsilon }
\notag \\
u-u^{\varepsilon,3} &=g && \text{ on }\partial \Omega
\notag
\end{alignat}
Apart from the difference between $\tilde u^*$ and $u^{\varepsilon,3}$, this is a
Stokes system and we proceed with bounding the norms of its solution in the
standard manner, \emph{cf.} \cite{GiraultRaviart}, using the inf-sup Lemmas 
\ref{infsupq} and \ref{infsupb}. Indeed, Lemma~\ref{infsupq} assures that there
exists $v_{p}\in H_{0}^{1}(\Omega^{\varepsilon })^{2}$ such that
\begin{equation*}
\Div v_{p}={G}_{\varepsilon }\mbox{ and }|v_{p}|_{H^{1}(\Omega^{\varepsilon
})}\leq \frac{C}{\varepsilon }\Vert{G}_{\varepsilon }\Vert_{L^2(\Omega)}\leq
C\varepsilon ^{\frac{3}{2}}\Vert f-\nabla p^{\ast }\Vert_{H^{2}(\Omega) \cap C^{1}(\bar\Omega)}\,.
\end{equation*}
Recall that $O^\delta$ with $\delta=c\varepsilon$ as introduced above, does not intersect $B^\eps$. Then, in view of the definition of $g$ (\ref{defg}) and equations \eqref{darcy}--\eqref{DarcyB}, Lemma~\ref{infsupb} assures that there exists $v_{b}\in H^{1}(\Omega^{\varepsilon })^{2}$ supported in $O^\delta$ and thus
vanishing on $B^\eps$ such that $\Div v_{b}=0$ on $\Omega^{\varepsilon }$,
\begin{equation*}
v_{b}=g\mbox{ on }\partial \Omega \mbox{ and }|v_{b}|_{H^{1}(\Omega
_{\varepsilon })}\leq \frac{C}{\sqrt{\varepsilon }}\Vert g\Vert_{C^{1}(\bar\Omega)}
\leq C\varepsilon ^{\frac{3}{2}}\Vert f-\nabla p^{\ast
}\Vert_{C^{1}(\bar\Omega)}
\end{equation*}

Set $v=u-u^{\varepsilon,3}-v_{p}-v_{b}$ and observe that $v\in H_{0}^{1}(\Omega
_{\varepsilon })^{2}$ and $\Div v=0$ on $\Omega^{\varepsilon }$.
Multiplying \eqref{StErr} by $v$ and integrating over $\Omega^{\varepsilon }$
by parts yields
\begin{equation*}
\int_{\Omega^{\varepsilon }}\nabla (u-\tilde u^{\ast }):\nabla v=\int_{\Omega
^{\varepsilon }}F_{\varepsilon }\cdot v\leq \Vert F_{\varepsilon
}\Vert_{L^2(\Omega)}\Vert v\Vert_{L^2(\Omega^{\varepsilon })}\leq C\varepsilon
\Vert F_{\varepsilon }\Vert_{L^2(\Omega)}|v|_{H^{1}(\Omega^{\varepsilon })}\,.
\end{equation*}
We have used here Poincar\'{e} inequality (\ref{Poin}) with $\phi=v$.  
Thus,
\begin{align*}
|v|_{H^{1}(\Omega^{\varepsilon })}^{2}& \leq C\varepsilon \Vert F_{\varepsilon
}\Vert_{L^2(\Omega)}|v|_{H^{1}(\Omega^{\varepsilon })}-\int_{\Omega
_{\varepsilon }}\nabla (u^{\varepsilon,3}+v_{p}+v_{b}-u^*):\nabla v \\
& \leq \left( C\varepsilon \Vert F_{\varepsilon }\Vert_{L^2(\Omega)}
+|\tilde u^{\ast}-u^{\varepsilon,3}|_{H^{1}}
+|v_{p}|_{H^{1}}+|v_{b}|_{H^{1}}\right) |v|_{H^{1}(\Omega^{\varepsilon
})}
\\
&\leq C\varepsilon ^{\frac{3}{2}}\Vert f-\nabla p^{\ast }\Vert_{H^{2}(\Omega) \cap C^{1}(\bar\Omega)}|v|_{H^{1}(\Omega^{\varepsilon })}\,,
\end{align*}
which follows from (\ref{eq:homogenization partial residual}), the above estimates on $v_p$ and $v_b$ and from the explicit expression
\begin{multline*}
\tilde u^{\ast}-u^{\varepsilon,3}=
\varepsilon^2 (\tilde w_i -w_i)_\varepsilon\left( f_i - \partial_i p^* \right)\\
-\varepsilon^{3}(\psi_{i})_{\varepsilon
}(\nabla ^{\perp}\eta ^{\varepsilon })(f_{i}-\partial_{i}p^{\ast })
+\varepsilon ^{3}(\gamma_{ij})_{\varepsilon }(1-\eta ^{\varepsilon
})\partial_{j}(f_{i}-\partial_{i}p^{\ast })\,,
\end{multline*}
which easily entails $|\tilde u^{\ast}-u^{\varepsilon,3}|_{H^{1}(\Omega^\eps)}
\leq C\varepsilon^2\Vert f-\nabla p^{\ast}\Vert_{H^{1}(\Omega)}$. This proves
\[
|v|_{H^{1}(\Omega_{\varepsilon })}\leq C\varepsilon ^{\frac{3}{2}}\Vert f-\nabla p^{\ast
}\Vert_{H^{2}(\Omega) \cap C^{1}(\bar\Omega)}\,,
\]
and consequently \eqref{errH1} by the triangle inequality. The $L^2$ estimate \eqref{errL2}
follows thanks to \eqref{Poin}. 

To prove the remaining estimate for
pressure \eqref{errP}, we  take $v\in H_{0}^{1}(\Omega^{\varepsilon })$ such that $\Div 
v=p-\tilde p^{\varepsilon ,1}$ as constructed in Lemma~\ref{infsupq}, multiply \eqref{StErr} by $v$ and integrate by parts 
\begin{multline*}
\int_{\Omega^{\varepsilon }}(p-\tilde p^{\varepsilon ,1})^2=
\int_{\Omega^{\varepsilon }}(p-\tilde p^{\varepsilon ,1})\Div v=\int_{\Omega
^{\varepsilon }}F_{\varepsilon }\cdot v-\int_{\Omega^{\varepsilon }}\nabla
(u-\tilde u^{\ast }):\nabla v
\\
\leq C\varepsilon ^{\frac{3}{2}}\Vert f-\nabla p^{\ast }\Vert_{H^{2}(\Omega) \cap C^{1}(\bar\Omega)}|v|_{H^{1}(\Omega ^{\varepsilon })}
\leq C\varepsilon ^{\frac{1}{2}} \Vert f-\nabla p^{\ast }\Vert_{H^{2}(\Omega) \cap C^{1}(\bar\Omega)}\Vert p-\tilde p^{\varepsilon ,1} \Vert_{L^2(\Omega^\varepsilon)}
\end{multline*}
using the estimate in Lemma~\ref{infsupq}. Thus, by triangle inequality,
\[
    \Vert p-p^\ast \Vert_{L^2(\Omega^\varepsilon)} \le
    \Vert p-\tilde p^{\varepsilon ,1} \Vert_{L^2(\Omega^\varepsilon)} +
    \Vert \tilde p^{\varepsilon ,1}-p^\ast \Vert_{L^2(\Omega^\varepsilon)} \le
    C\varepsilon ^{\frac{1}{2}}\Vert f-\nabla
    p^{\ast }\Vert_{H^{2}(\Omega) \cap C^{1}(\bar\Omega)}
\]
since~$(\tilde p^{\varepsilon,1} - p^*)$ term is of order $\eps$ as seen from (\ref{eq:definition p epsilon one}).
\qquad
\endproof

\section{Technical lemmas}
\label{sec:preliminaries}
We assume implicitly in this section that mesh $\TH$ is quasi-uniform, as described in the beginning of Section \ref{sec:crmsfem} and that Assumptions \ref{AssWeights2}-\ref{AssWeights3} and \ref{AssPeriod} are valid.  The weights $w_i$ are assumed to be chosen as in (\ref{CR3}), i.e. we only study the CR3 variant of the method.

\subsection{Some lemmas borrowed from the usual finite element theory}
\label{ssec:ingredients}
\quad

\begin{lemma}\label{lem:traceH}
For all~$T\in \mathcal T_H$, all the edges $E \subset \partial T$ and all $v\in H^{1}(T)$
\begin{equation}
\|v\|_{L^2(E)}^{2}\leq
C\left(H^{-1}\|v\|_{L^2(T)}^{2}+H\|\nabla v\|_{L^2(T)}^{2}\right)\,.
\label{eq:trace1}
\end{equation}
\end{lemma}
\begin{proof}
  This is the standard trace inequality properly scaled to a domain of diameter $\sim H$, cf. \cite[Section 4.2]{MsFEMCR1}).
  \qquad
\end{proof}

\begin{lemma}\label{lem:P0_EF}
Let $\Pi_{H}$ be the $L^2(\Omega)$-orthogonal
projection on the space of piecewise constant functions on $\TH$. For any $f\in H^1(\Omega)$
\begin{equation}
\Vert f-\Pi_{H}f\Vert_{L^2(\Omega)}\leq CH |f|_{H^1(\Omega)} 
\label{eq:P0_EF}
\end{equation}
with a constant $C>0$ depending only on the regularity of $\TH$.
\end{lemma}
\begin{proof}
This is a standard finite element interpolation result. It is proven by a Poincar\'e inequality on the reference element and scaling.
\end{proof}

\begin{lemma}\label{lem:P1_EF} There exists a bounded linear operator~$I_H: H^2(\Omega) \to H^1(\Omega)$ such that $I_Hv$ is a polynomial of degree $\le 1$ on any edge $E\in\mathcal{E}_H$ for any $v\in H^2(\Omega)$ and 
    \begin{equation*}\label{eq:P1_EF}
    |I_H v-v|_{H^{1}(\Omega)}\leq CH|v|_{H^{2}(\Omega)}
    \end{equation*}
and    
    \begin{equation*}\label{eq:P1_EFl2}
    \|I_H v-v\|_{L^{2}(\Omega)}\leq CH|v|_{H^{1}(\Omega)}
    \end{equation*}
with a constant $C>0$ depending only on the regularity of $\TH$.
\end{lemma}
\begin{proof}
One can simply take $I_H$ as the usual Cl\'ement interpolation operator on $P_1$ finite elements if $\TH$ is a triangular mesh. Otherwise,
we consider~$\widehat{\mathcal T}_H$ a submesh of~$\mathcal T_H$ which consists
of triangles only. To construct~$\widehat{\mathcal T}_H$, one only needs
to remesh the reference element~$\overline T$ in triangles, without adding nodes on~$\partial\overline T$. 
Applying the mapping~$K$ on each element of $\TH$ one obtains then $\widehat{\mathcal T}_H$. We can now define $I_H$ as the Cl\'ement interpolation operator on $P_1$ finite elements on $\widehat{\mathcal T}_H$. \qquad
\end{proof}

\subsection{Lemmas related to perforated domains and oscillating functions}

\begin{lemma}
\label{lem:tracp} Suppose $H\ge\gamma\varepsilon$ with some big enough~$\gamma$.
Let $T\subset {\mathcal{T}_{H}}$ and take any $v\in H^{1}(T)$
vanishing on $B^\eps\cap T$. Then,
\begin{equation}
\Vert v\Vert_{L^2(\partial T)}\le C\sqrt{\varepsilon}|v|_{H^{1}(T)}\,.
\label{eq:tracp}
\end{equation}
The constants $\gamma>0$ and $C>0$ here depend only on the regularity
of mesh ${\mathcal{T}_{H}}$ and on the perforation pattern $B$.
\end{lemma}

\begin{proof}
\begin{figure}[tbp]
\centering
\begin{tikzpicture}[scale=2]
    \draw[gray!50,step=0.25cm] (-2.1,-1.7) grid (2.1,1.7);


    \draw (2,0) -- (1.5,1.5) -- (0.7,1.6) -- (-0.7,1.4) -- (-2,0.55) -- (-1.5,-1)
    -- (-0.3,-1.5) --
    (1,-1) -- (2,0) node[right] {$\Omega^\varepsilon$};

    \coordinate (A) at (1,0);
    \coordinate (B) at (-0.55,1);
    \coordinate (C) at (-0.75,-0.85);

    \pgfmathsetmacro{\foo}{0.85/0.25/7};

    \coordinate (M1) at ($(C) + (0.25,0.25*3/7)$);
    \coordinate (M2) at ($(C) + (2*0.25,2*0.25*3/7)$);
    \coordinate (M3) at ($(C) + (3*0.25,3*0.25*3/7)$);

    \foreach \i in {0,...,7} {
        \pgfmathsetmacro{\pxx}{\i*0.25};
        \pgfmathsetmacro{\pyy}{\i*0.25*\foo};
        \coordinate (M\i) at ($(C) + (\pxx,\pyy)$);
        \draw (M\i) circle (0.01);
    }

    \begin{scope}[every node/.style={scale=0.8}]
    \foreach \i in {0,...,6} {
        \pgfmathsetmacro{\ip}{\i+1};
        \draw (M\i) -- node[below] {$S^\i$} (M\ip);
    }
    \draw node[right] at (0.181+0.25*3,0.125) {$S^7$};
    \draw node[label={[rotate=-35]$\dots$}] at (0.25+0.25*2,0.25-0.063) {};
    \draw node[label={[rotate=-10]$\vdots$}] at (0.15-0.25*4,0.025-2*0.25) {};
    \draw node at (0.125-0.25*4,0.125-3*0.25) {$S^n$};
    \end{scope}

    \coordinate (C1) at ($(C) + (0.125, 0.225)$);
    \coordinate (C2) at ($(C) + (0.125+1*0.25, 0.225+1*0.25)$);
    \coordinate (C3) at ($(C) + (0.125+2*0.25, 0.225+1*0.25)$);
    \coordinate (C4) at ($(C) + (0.125+3*0.25, 0.225+2*0.25)$);
    \coordinate (C5) at ($(C) + (0.125+4*0.25, 0.225+2*0.25)$);
    \coordinate (C6) at ($(C) + (0.125+4*0.25, 0.225+3*0.25)$);
    \coordinate (C7) at ($(C) + (0.125+5*0.25, 0.225+3*0.25)$);

    \draw[pattern=north east lines] (M0) -- (C1) -- (M1);
    \draw[pattern=north east lines] (M1) -- (C2) -- (M2);
    \draw[pattern=north east lines] (M2) -- (C3) -- (M3);
    \draw[pattern=north east lines] (M3) -- (C4) -- (M4);
    \draw[pattern=north east lines] (M4) -- (C5) -- (M5);
    \draw[pattern=north east lines] (M5) -- (C6) -- (M6);
    \draw[pattern=north east lines] (M6) -- (C7) -- (M7);

    \foreach \i in {1,...,7}
        \draw[color=white,fill=white] (C\i) circle (0.05);

        \draw (A) -- (B) node[above left] {$T$} -- (C) -- (A);

    \begin{scope}
        \clip (A) -- (B) -- (C) -- (A);

        \foreach \i in {0,...,7} {
            \foreach \j in {0,...,7} {
                \pgfmathsetmacro{\pxx}{\i*0.25};
                \pgfmathsetmacro{\pyy}{\j*0.25};
                \coordinate (MM\i) at ($(-0.625,-0.625) + (\pxx,\pyy)$);
                \draw (MM\i) circle (0.05);
            }
        }
    \end{scope}

\end{tikzpicture}
\caption{Partition of the boundary of~$T$ (bottom edge only, for clarity)}
\label{fig:Pict1}
\end{figure}
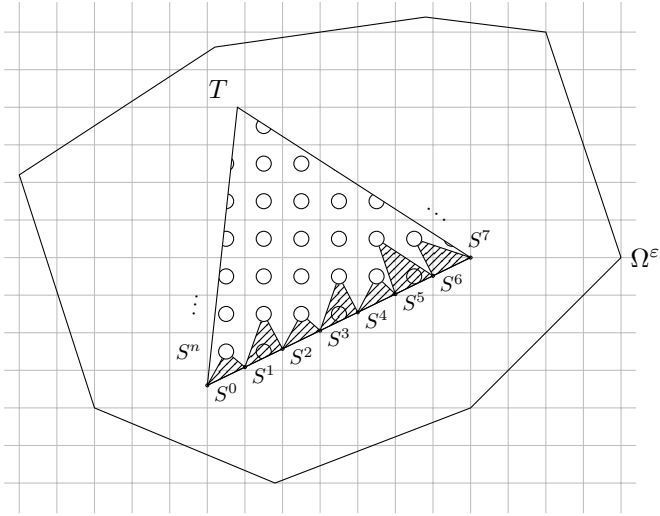
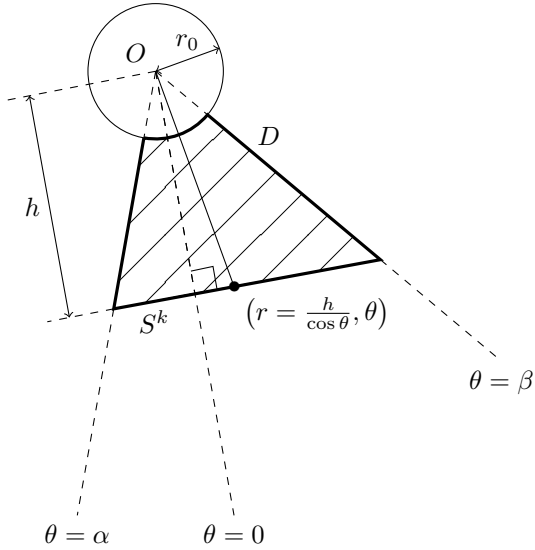
\begin{figure}
\centering
\begin{tikzpicture}[scale=3]
    \coordinate (A) at (-100:1.07);
    \coordinate (M) at (-80:1);
    \coordinate (B) at (-40:1.3);
    \coordinate (K) at (-70:1.02);

    \coordinate (Ap) at (-100:0.3);
    \coordinate (Bp) at (-40:0.3);

    \node[above left] (orig) at (0,0) {$O$};

    \draw (0,0) circle (0.3);
    \draw[<->] (0,0) -- node[above] {$r_0$} ++ (20:0.3);

    \draw[dashed] (A) -- (-100:2) node[below] {$\theta=\alpha$};
    \draw[dashed] (0,0) -- (-80:2) node[below] {$\theta = 0$};
    \draw[dashed] (0,0) -- (-80:1);
    \draw[dashed] (B) -- (-40:2) node[below] {$\theta = \beta$};

    \draw[very thick, XHatchSize=18pt, pattern=XHatch] (Ap) -- (A) -- (B) -- (Bp) arc (-40:-100:0.3);

    \node (meuh) at (K) {$\bullet$};
    \draw[dashed] (0,0) -- (A);
    \draw[dashed] (0,0) -- (B);

    \draw (0,0) -- (K) node[below right]
    {$\left(r=\frac{h}{\cos\theta},\theta\right)$};

    \coordinate (Ah) at ($(A) + (-169:0.3)$);
    \coordinate (Oh) at ($(0,0) + (-169:0.67)$);

    \draw[thin] (M) ++ (12:0.1) -- ++ (100:0.1) -- ++ (192:0.1);
    \draw (A) -- node[below] {$S^k$} (M);

    \draw[dashed] (0,0) -- (Oh);
    \draw[dashed] (A) -- (Ah);

    \draw[<->] ($(Oh) + (11:0.1)$) -- node[left] {$h$} ($(Ah) + (11:0.1)$);

    \draw (A) -- (M) -- (B);

    \node[right] (midD) at (-35:0.5) {$D$};
\end{tikzpicture}
\caption{Local coordinates system associated to some~$S^{k}$}
\label{fig:Pict2}
\end{figure}
We can safely suppose that the perforation pattern $B$ contains a disc of
radius $c_1>0$. It means that each perforation $\varepsilon B_k$, $
k\in\mathbb{Z}^2$ contains a disc of radius $r_0=c_1\varepsilon$. As shown
in \figref{fig:Pict1}, the boundary $\partial T$ can be decomposed into non
overlapping segments $S^0, S^{1}, S^{2}\ldots$ such that
each segment $S$ lies at a distance no greater than $c_2\varepsilon$
from the center of a disc of radius $r_0$ which lies
completely inside $B^\varepsilon\cap T$. In order to do this, we should
suppose that the mesh cell is big enough, hence the restriction $
H\ge\gamma\varepsilon$. Thus, to each segment
$S$ we associate a disc of radius $r_0$ centered at a point $O$ and a
``sector'' $D$ (see \figref{fig:Pict2}) which is bounded by two lines
intersecting at $O$,
by $S$ itself and by a portion of the circle centered at $O$.

Let us fix a segment $S$ as above and introduce properly shifted and rotated
polar coordinates $(r,\theta)$ such that $r=0$ corresponds to the disc
center $O$ and $\theta=0$ corresponds to the direction normal to $S$, \emph{cf.}
\figref{fig:Pict2}. The segment $S$ is parameterized in these coordinates as
\begin{equation*}
\theta\in[\alpha,\beta]\mapsto X_{\theta}:=
\begin{pmatrix}
r_{\theta} \\
\theta
\end{pmatrix}
\mbox{ with }r_{\theta}=\frac{h}{\cos\theta}\,,
\end{equation*}
where $h$ is the minimal distance from point $O$ to the line containing $S$
and $\alpha<0<\beta$. A simple geometrical calculation yields
\begin{equation*}
|dX_{\theta}|=\frac{h}{\cos^{2}\theta}d\theta\,,
\end{equation*}
so that
\begin{equation*}
\int_{S}v^{2}=\int_{\alpha}^{\beta}v^{2}(r{}_{\theta},\theta)\frac{h}{
\cos^{2}\theta}d\theta\,,
\end{equation*}
where we write $v$ as a function of polar coordinates $(r,\theta)$. Since $v$
vanishes in the holes, we have $v(r{}_{0},\theta)=0$ and
\begin{align*}
\int_{S}v^{2} & =\int_{\alpha}^{\beta}\left(\int_{r_{0}}^{r_{\theta}}\frac{
\partial v}{\partial r}(r,\theta)dr\right)^{2}\frac{h}{\cos^{2}\theta}d\theta
\\
& \le\int_{\alpha}^{\beta}\left(\int_{r_{0}}^{r_{\theta}}|\nabla
v|^{2}(r,\theta)dr\right)(r_{\theta}-r_{0})\frac{r_{\theta}^{2}}{h}d\theta \\
& \le\left(\max_{\theta\in[\alpha,\beta]}\frac{(r_{\theta}-r_{0})r_{
\theta}^{2}}{hr_{0}}\right)\int_{\alpha}^{\beta}\int_{r_{0}}^{r_{\theta}}|
\nabla v|^{2}(r,\theta)r\, dr\, d\theta\le C\varepsilon\int_{D}|\nabla v|^{2}\,,
\end{align*}
with some constant $C>0$. Indeed, under our geometrical assumptions we have $
h\ge r_{0}\ge c_{1}\varepsilon$, $r_{\theta}\le c_{2}\varepsilon$ so that $
\frac{(r_{\theta}-r_{0})r_{\theta}^{2}}{hr_{0}}\le\frac{c_{2}^{3}}{c_{1}^{2}}
\varepsilon$. Now, summing up over all the segments composing $\partial T$
and noting that the sector $D$ corresponding to such a segment $S$ is inside
the cell $T$ and, moreover, for any two segments $S,S^{\prime }$ the
corresponding sectors $D,D^{\prime }$ do not intersect, yields (\ref
{eq:tracp}).
\end{proof}

\begin{lemma}[Poincar\'e inequality on a perforated mesh cell]
    \label{lem:poincare perforated T}
    Suppose $H\ge\gamma\eps$ with $\gamma$ from Lemma \ref{lem:tracp}. Then, for any $T\in\TH$ and any $v\in H^1(T)$ vanishing on $B^\eps\cap T$
    \begin{equation}
            \Vert v \Vert_{L^2(T)} \leq \varepsilon\,C \left|v\right|_{H^1(T)}
        \label{eq:poincare perforated T}
    \end{equation}
    with some positive $\eps$- and $h$-independent constant~$C$.
\end{lemma}
\begin{proof}
Applying a Poincar\'e inequality on the reference cell $Y$ with the hole $B$ and then rescaling to the cells of size $\eps$ gives 
$$
            \Vert v \Vert_{L^2(Y_k)} \leq \varepsilon\,C \left|v\right|_{H^1(Y_k)}
$$
for any perforated cell $Y_k$, $k\in\Z^2$ and any $v\in H^1(Y_k)$ vanishing on $B_k$. 
Let $\mathcal{I}(T)\subset\Z^2$ be the set of indexes corresponding to the cells inside $T$ and assume that the boundary of $T$ is composed of $m$ edges $E_1,\ldots,E_m$. One can then introduce $m$ rectangles $\Pi_1,\ldots,\Pi_m$, each $\Pi_i$ with base  $E_i$ and of width (in the direction perpendicular to $E_i$) $\le c\eps$ with some $H$-independent constant $c$, so that
$$
T\subset \cup_{k\in\mathcal{I}(T)}Y_k\cup\Pi_1\cup\cdots\cup\Pi_m.
$$
We can also safely assume that every point in $T$ is covered by at most 3 subsets on the right-hand send of the inclusion above, i.e. at most by a cell $Y_k$ and by two rectangles $\Pi_i$.

Let us introduce the Cartesian coordinates $( \xi, \eta)$ on rectangle $\Pi_i$
so that $\eta = 0$, $\xi \in [0, |E_i|]$ corresponds to $E_i$ and the
coordinate $\eta$ varies from 0 to some $h_i \leq c \varepsilon$ on $\Pi_i$.
Assuming that $v$ is extended from $\Pi_i \cap T$ to the whole $\Pi_i$ so that
the $H_1$ norm of $v$ over $\Pi_i$ remains bounded via that over $\Pi_i \cap
T$, we calculate
\begin{multline*}
 \| v \|_{L^2 ( \Pi_i)}^2 = \int_0^{|E_i|} \int_0^{h_i} v^2 ( \xi, \eta) d
   \eta d \xi
\\   
   = h_i \int_0^{|E_i|} v^2 ( \xi, 0) d \xi + \int_0^{|E_i|}
   \int_0^{h_i} \int_0^{\eta} 2 v ( \xi, s) \partial_{\eta} v ( \xi, s)
   ds d\eta d\xi 
\\ 
 \leq c \varepsilon \| v \|^2_{L^2 ( E_i)} + \frac{1}{2} \| v \|_{L^2 (
   \Pi_i)}^2 + C \varepsilon^2 | v |_{H^1 ( \Pi_i)}^2 . 
\end{multline*}
Thus, 
\[ \| v \|_{L^2 ( \Pi_i \cap T)}^2 \leq C ( \varepsilon \| v \|^2_{L^2 ( E_i)}
   + \varepsilon^2 | v |_{H^1 ( \Pi_i \cap T)}) . \]
Summing over all the cells $Y_k$, $k \in \mathcal{I} ( T)$ and all the
rectangles $\Pi_i$ and reminding that each point of $T$ is covered by at most 3 such sets, gives
\[ \| v \|_{L^2 ( T)}^2 \leq C ( \varepsilon \| v \|^2_{L^2 ( \partial T)} +
   \varepsilon^2 | v |_{H^1 ( T)}) \]
which entails (\ref{eq:poincare perforated T}) thanks to Lemma \ref{lem:tracp}.
\end{proof}

\begin{lemma}[Poincar\'e inequality in $H^1$ - broken spaces]
    \label{lem:poincare perforated broken}
  For any~$v \in V_H^{{ext}}$
  \begin{equation}
    \|v\|_{L^2 (\Omega^{\varepsilon})} \leq \varepsilon \hspace{0.17em} C | v
    |_{H^1 (\Omega^{\varepsilon})} . \label{eq:poincare perforated broken}
  \end{equation}
  wth some positive $\varepsilon$-independent constant~$C$.
\end{lemma}

\begin{proof}
  We distinguish two cases: $H \ge \gamma \varepsilon$ with $\gamma$ from
  Lemma \ref{lem:tracp} and $H < \gamma \varepsilon$. In the first case, the
  current lemma is a simple corollary of the previous one obtained by
  summing (\ref{eq:poincare perforated T}) over all the mesh cells. We thus
  assume from now on $H < \gamma \varepsilon$. Borrowing from \cite{Chainais13} the idea of
  using an embedding theorem for BV spaces (the functions of bounded
  variation), we can write on each cell $Y_k$, $k \in \mathbb{Z}^2$ (of size
  $\varepsilon$, with the perforation $B_k$ inside) and any $v \in V_H^{ext}$
  extended by 0 outside $\Omega$
  \begin{equation}\label{BVimb}
 \| v \|_{L^2 (Y_k)} \le C\operatorname{TV}_{Y_k} (v) :=
     C\sup_{\varphi \in C^1_C (Y_k), | \varphi | \le 1
     \text{ on }Y_k} \int_{Y_k} v\Div \varphi 
  \end{equation}
  We have applied here Theorem 2 from \cite{Chainais13} the proof of which can be found in \cite[Chapter 3]{BV00}.\footnote{
  We recall that the ambient dimension is assumed equal to 2 in this paper. Were we interested in the case of a perforated domain in $\R^d$ with $d>2$, we would have the norm of $L^{d/(d-1)}$ rather than $L^2$ in the left-hand side of (\ref{BVimb}). A proof of (\ref{eq:poincare perforated broken}) could be then performed by first applying (\ref{BVimb}) to $|v|^\alpha$ with $\alpha=2\frac{d-1}{d}$ rather than to $v$.}
  Note that we can use the semi-norm $\operatorname{TV}_{Y_k}$ of the BV space
  since $v$ vanishes on the perforation $B_k$. The constant
  $C$ is in principle domain dependent but it can be considered
  $\varepsilon$-independent in our case. Indeed, the inequality above is
  invariant under scaling $x \mapsto (x - x_k) / \varepsilon$ so that the
  value of $C$ can be taken as that on the reference cell $Y$ with its
  reference perforation $B$.
  
  Integration by parts and Cauchy-Schwarz inequality gives for any $\varphi
  \in C^1_C (Y_k)$ such that $| \varphi | \le 1$ on $Y_k$
  \begin{align*}
   \left| \int_{Y_k} v\Div \varphi \right| &= \left| - \int_{Y_k
     \backslash\mathcal{E}_H} \nabla v \cdot \varphi + \sum_{E \in
     \mathcal{E}_H} \int_{Y_k \cap E} [[v]] n \cdot \varphi \right| \\ 
     &\le \varepsilon | v |_{H^1 (Y_k)} + \left( \sum_{E \in \mathcal{E}_H}
     \| [[v]] \|^2_{L^2 (Y_k \cap E)} \right)^{\frac{1}{2}} \left( \sum_{E \in
     \mathcal{E}_H} | Y_k \cap E | \right)^{\frac{1}{2}} \\
    &\le \varepsilon | v |_{H^1 (Y_k)} + C
     \frac{\varepsilon}{\sqrt{H}} \left( \sum_{E \in \mathcal{E}_H} \| [[v]]
     \|^2_{L^2 (Y_k \cap E)} \right)^{\frac{1}{2}} 
  \end{align*}
  Indeed, the number of mesh edges intersecting $Y_k$ is of the order of
  $\frac{\varepsilon^2}{H^2}$ and the length of each edge is smaller than $H$
  so that $\sum_{E \in \mathcal{E}_H} | Y_k \cap E | \le C
  \frac{\varepsilon^2}{H}$. Taking the supremum over $\varphi$ gives
  \[ \| v \|_{L^2 (Y_k)}^2 \le C \left( \varepsilon^2 | v |^2_{H^1
     (Y_k)} + \frac{\varepsilon^2}{H} \sum_{E \in \mathcal{E}_H} \| [[v]]
     \|^2_{L^2 (Y_k \cap E)} \right) \]
  Summing this over all the cells $Y_k$ gives
  \[ \| v \|^2_{L^2 (\Omega)} \le C \left( \varepsilon^2 | v |^2_{H^1
     (\Omega)} + \frac{\varepsilon^2}{H} \sum_{E \in \mathcal{E}_H} \| [[v]]
     \|^2_{L^2 (Y_k)} \right) \]
  By the trace inequality $\| [[v]] \|^2_{L^2 (Y_k)} \le {CH} | v
  |^2_{H^1 (\omega_E)}$, this entails the desired result
  \[ \| v \|^2_{L^2 (\Omega)} \le C \varepsilon^2 \left( | v |^2_{H^1
     (\Omega)} + \sum_{E \in \mathcal{E}_H} | v |^2_{H^1 (\omega_E)} \right)
     \le C \varepsilon^2 | v |^2_{H^1 (\Omega)} \]  
\end{proof}

The proof of the following lemma uses extensively the results and notations on homogenization from Section \ref{sec:homog}. 
It will be the principal ingredient of the proof of Theorem \ref{theo:main}.

\begin{lemma}
\label{lem:brique}Let $u,p$ be the solution to the Stokes system
\eqref{genP}--\eqref{genB} and set
$p=p^{\ast }+p^{\prime }$ where $p^{\ast }$ is the solution to the Darcy
problem \eqref{darcy}--\eqref{DarcyB}.
Under the same assumptions as those of Theorem~\ref{theo:main}
with $\gamma $ from Lemma~\ref{lem:tracp}, we have, for any $v\in \ZHext:=
\{v\in\VHext:\Div v|_T=0\ \forall T\in\TH\}$
\begin{equation}
\left\vert \sum_{T\in \mathcal{T}_{H}}\int_{\partial T\cap \Omega^{\varepsilon }}(\left(\nabla u\right)n-p^{\prime }n)\cdot v\right\vert \leq
C\varepsilon \left( \sqrt{\varepsilon }+\sqrt{\frac{\varepsilon }{H}}\right)
\left| v\right|_{H^{1}(\Omega^{\varepsilon })}\,\Vert f-\nabla p^{\ast
}\Vert_{H^{2}(\Omega) \cap C^{1}(\bar\Omega)},  \label{eq:brique}
\end{equation}
where the constant $C$ is independent of $H$, $\varepsilon $, $f$ and $v$.
\end{lemma}

\begin{proof}
Using the divergence theorem on any $T\in \mathcal{T}_{H}$ and reminding
\eqref{eq:definition p epsilon one} and
$\Div v=0$ on $T$, we observe that
\begin{align}
&\sum_{T\in \mathcal{T}_{H}}\int_{\partial T\cap \Omega^{\varepsilon
}}(\left(\nabla u\right)n-p^{\prime }n)\cdot v
\notag\\
&=\sum_{T\in \mathcal{T}
_{H}}\int_{\Omega^{\varepsilon }\cap T}\nabla u:\nabla v-\int_{\Omega
_{\varepsilon }\cap T}(f-\nabla p^{\ast })\cdot v
\notag\\
&=\sum_{T\in \mathcal{T}_{H}}\left[ \int_{\Omega^{\varepsilon }\cap
T}(\nabla u-\nabla \tilde u^{\ast }):\nabla v+\int_{\Omega^{\varepsilon }\cap
T}(\nabla \tilde u^{\ast }-\left(\tilde p^{\varepsilon ,1}-p^{\ast }\right)I):\nabla v\right.
\notag\\
&\left.\quad\quad\quad-\int_{\Omega^{\varepsilon }}(f-\nabla p^{\ast })\cdot v\right]
\notag\\
&=\sum_{T\in \mathcal{T}_{H}}\int_{\Omega^{\varepsilon }\cap T}(\nabla
u-\nabla \tilde u^{\ast }):\nabla v+\sum_{T\in \mathcal{T}_{H}}\int_{\Omega
_{\varepsilon }\cap \partial T}\left(\left(\nabla \tilde u^{\ast }\right)n-\left(\tilde p^{\varepsilon
,1}-p^{\ast }\right)n\right)\cdot v
\notag\\
&\quad\quad\quad-\sum_{T\in \mathcal{T}_{H}}\int_{\Omega^{\varepsilon }\cap T}(f+\Delta
\tilde u^{\ast }-\nabla \tilde p^{\varepsilon ,1})\cdot v \,. \label{brique1}
\end{align}
The first term in the sum above can be bounded by
$C\varepsilon \sqrt{\varepsilon }\Vert
f - \nabla p^\ast\Vert_{\espHC}|v|_{H^{1}(\Omega^{\varepsilon })}$,
using the homogenization estimate \eqref{errL2}. We turn
now to the second term in \eqref{brique1}.

Using Lemmas~\ref{lem:traceH} and~\ref{lem:tracp}
and the fact that $\tilde w_{i}$, $\tilde \pi_{i}$ and $\nabla
\tilde w_{i}$ are uniformly bounded, we have for any $T\in \mathcal{T}_{H}$
\begin{multline*}
\quad\left\vert \int_{\Omega^{\varepsilon }\cap \partial T}((\nabla \tilde u^{\ast
})n-(\tilde p^{\varepsilon ,1}-p^{\ast })n)\cdot v\right\vert
\\
 =\left\vert
\int_{\partial T\cap \Omega^{\varepsilon }}\left[ \varepsilon (\nabla
\tilde w_{i})_{\varepsilon }n(f_{i}-\partial_{i}p^{\ast })+\varepsilon
^{2}(\tilde w_{i})_{\varepsilon }\cdot \nabla (f_{i}-\partial_{i}p^{\ast
})n-\varepsilon (\tilde \pi_{i})_{\varepsilon }(f_{i}-\partial_{i}p^{\ast })
\right] \cdot v\right\vert
\\
 \leq C \Vert v\Vert_{L^2(\partial T)}
 [\varepsilon \Vert f-\nabla p^{\ast }\Vert_{L^{2}(\partial T)} + \varepsilon^2 \Vert \nabla(f-\nabla p^{\ast })\Vert_{L^{2}(\partial T)} ]
\\
 \leq C\varepsilon \sqrt{\frac{\varepsilon }{H}}|v|_{H^{1}(T)}\,\Vert
f-\nabla p^{\ast }\Vert_{H^{2}(T)}\,.
\end{multline*}
Now, summing up over all the cells and using the discrete Cauchy-Schwarz
inequality yields
\begin{equation*}
\left\vert \sum_{T\in\TH} \int_{\Omega^{\varepsilon }\cap \partial T}((\nabla u^{\ast
})n-(\tilde p^{\varepsilon ,1}-p^{\ast })n)\cdot v\right\vert \leq C\varepsilon
\sqrt{\frac{\varepsilon }{H}}\left| v\right|_{H^{1}(\Omega^{\varepsilon
})}\,\Vert f-\nabla p^{\ast }\Vert_{H^{2}(\Omega)}.
\end{equation*}
To bound the third term in \eqref{brique1}, we
recall the definition of~$F_\varepsilon$~\eqref{eq:stokes expansion residual}
and observe that
\begin{equation*}
f+\Delta \tilde u^{\ast }-\nabla \tilde p^{\varepsilon ,1} =  F_\varepsilon
\end{equation*}
Thus, using the estimate of $F_\eps$ and Poincar\'e inequality from Lemma \ref{lem:poincare perforated broken},
\begin{align*}
\left\vert \sum_{T\in \mathcal{T}_{H}}\int_{\Omega_{\varepsilon }\cap
T}(f+\Delta \tilde u^{\ast }-\nabla \tilde p^{\varepsilon ,1}):v\right\vert
&\leq C\varepsilon \sqrt{\eps}\Vert f-\nabla p^{\ast
}\Vert_{\espHC}\left| v\right|_{H^{1}(\Omega^{\varepsilon })}\,.
\end{align*}
Summing up the bounds for all the three terms in \eqref{brique1} yields (\ref
{eq:brique}).
\end{proof}

\section{Proof of Theorem \protect\ref{theo:main}}

\label{sec:proof}

We note first of all that error estimate (\ref{MainErrEst}) is trivial if $H$ is of order $\eps$ or smaller. Indeed, if $H\leq \gamma \varepsilon $, then (\ref{MainErrEst}) is reduced to 
\begin{multline}\label{MainErrEstTriv}
|{u}-{u}_{H}|_{H^{1}(\Omega)}+\varepsilon \Vert p-p_{H}\Vert_{L^2(\Omega)}
\leq C\varepsilon 
(\Vert f\Vert_{\espHC}+\Vert p^{\ast }\Vert_{H^2(\Omega)})\,
\end{multline}
with a constant $C$ depending on $\gamma$. But we have, in fact for any $\eps$ and $H$,
\begin{align*}
|{u}|_{H^{1}(\Omega)}+\varepsilon \Vert p\Vert_{L^2(\Omega)}
&\leq C\varepsilon \Vert f\Vert_{L^{2}(\Omega)} \,,
\\
|{u}_{H}|_{H^{1}(\Omega)}+\varepsilon \Vert p_{H}\Vert_{L^2(\Omega)}
&\leq C\varepsilon \Vert f\Vert_{L^{2}(\Omega)} \,.
\end{align*}
These estimates for the velocity are easily obtained from the Poincar\'e inequality on the perforated domain $\Omega^\eps$ which is valid even for the broken $H^1$ Sobolev space, as proven in Lemma \ref{lem:poincare perforated broken}. As for the pressure,  
these are the standard bounds for the solutions of saddle-point problems since the inf-sup property holds with a constant of order $\eps$ both on continuous and discrete levels, cf. Lemmas \ref{infsupq} and \ref{infsupD}. This clearly entails (\ref{MainErrEstTriv}) and consequently (\ref{MainErrEst}) if $H<\gamma\eps$.

We thus assume from now on $H\geq \gamma \varepsilon $ with $\gamma $ from
Lemma~\ref{lem:tracp} and use without further notice Lemmas \ref{lem:tracp}, \ref{lem:poincare perforated T}, \ref{lem:brique} from the previous section. Our error
estimate is essentially based on a Strang lemma for nonconforming finite element methods. It can be 
stated in our notations, recalling equation (\ref{MsFEMz}) for $u_H$, as
\begin{lemma}
\textbf{(\emph{e.g.}~\cite[Lemma~10.1.7]{brenner})} \label{lem:strang} Let $u$ be
the solution to~\eqref{genP}--\eqref{genB} and $u_{H}$ be the solution to~\eqref{MsFEMz}.
Then
\begin{equation}
\left| u-u_{H}\right|_{H^{1}}\leq \inf_{v\in \ZH} |u-v|_{H^1}+\sup_{v\in \ZH\setminus \{0\}}\frac{|a (u-u_{H},v)|}{|v|_{H^1}}.  \label{eq:strang}
\end{equation}
\end{lemma}

The first term in~\eqref{eq:strang} is the usual best approximation error
already present in the classical C\'{e}a Lemma. The second term of~\eqref{eq:strang}
is the \emph{nonconformity error}, that is, roughly speaking, how far~$\ZH$
is from the divergence free subspace of~$H_0^1(\Omega)$. 

To bound the first term in~\eqref{eq:strang}.
we recall that $u$ is the solution to problem~\eqref{genP}--\eqref{genB} and introduce
\begin{align*}
v_{H}(x) &=\sum_{E\in \mathcal{E}_{H}}\sum_{i=1}^3
  \left(\int_E u\cdot\psi_{E,i}\right) {\Phi}_{E,i}(x) , \\
q_{H}(x) &=\sum_{E\in \mathcal{E}_{H}}\sum_{i=1}^3
  \left(\int_E u\cdot\psi_{E,i}\right) {\pi}_{E,i}(x)
\end{align*}
with $\Phi_{E,i}$ and~$\pi_{E,i}$ defined in Lemma~\ref{PhiGen} with the weights $\omega_{E_i}$ chosen as in (\ref{CR3}).
Observe, for all edges $E\in \mathcal{E}_{H}$ and
all cells $T\in \mathcal{T}_{H}$, that
\begin{align}
\int_{E}v_{H} &= \int_{E}u\,,
\nonumber
\\
\int_{E}\psi_{E}v_{H}\cdot n_{E} & = 
\int_{E}\psi_{E}u\cdot n_{E}\,,
\nonumber
\\
\left(\nabla v_{H}\right)n-q_{H}n & = {a}
+b\,n_{E}\psi_{E}\quad\text{ on (each side of) $E$}
\label{eq:par_construction}
\\
& \qquad\text{ with $a\in \mathbb{R}^2,b\in \mathbb{R}$, }
\nonumber
\\
-\Delta v_{H}+\nabla q_{H} & = 0\text{ on $T\cap \Omega
_{\varepsilon }$}.
\nonumber
\end{align}
By construction, $v_{H}\in \VH{}$. Moreover, it is easy to see that $
v_{H}\in \ZH$. Indeed, for any $T\in T_{H}$ we have $\Div v_{H}=c_{T}$ on $
T\setminus B^{\varepsilon }$ with some constant $c_{T}$ and
\begin{equation*}
c_{T}|T\setminus B^{\varepsilon }|=\int_{T}\Div v_{H}=\int_{\partial
T}n\cdot v_{H}=\int_{\partial T}n\cdot u=0\,,
\end{equation*}
so that $c_{T}=0$.

We also have, setting $p=p^*+p'$, as in Lemma \ref{lem:brique}
\begin{equation}
\begin{aligned}
|u-v_H|_{H^{1}(\Omega^{\varepsilon })}^{2}
&=
\sum_{T\in \mathcal{T}
_{H}}\int_{\Omega_{\varepsilon }\cap T}\nabla (u-v_{H}):\nabla (u-v_H) 
\\
&\qquad -\sum_{T\in \mathcal{T}_{H}}\int_{\Omega_{\varepsilon
}\cap T}(p^{\prime }-q_{H})\Div(u-v_H)
\\
&=\sum_{T\in \mathcal{T}_{H}}\int_{\Omega^{\varepsilon }\cap T}(-\Delta
(u-v_{H})+\nabla (p^{\prime }-q_{H}))\cdot (u-v_H)
\\
&\quad+\sum_{T\in \mathcal{T}
_{H}}\int_{\partial T\cap \Omega^{\varepsilon }}(u-v_H) \ \cdot ((\nabla
u)n-p^{\prime }n)
\\
&\quad-\sum_{T\in \mathcal{T}_{H}}\int_{\partial T\cap \Omega
_{\varepsilon }}(u-v_H) \cdot ((\nabla v_{H})n-q_{H}n)\,.
\end{aligned}
\label{eq:decompo_error}
\end{equation}
We now successively bound the three terms of the right-hand side of~
\eqref{eq:decompo_error}.

\begin{itemize}
\item For the first term, we observe
that $\int_{E}s\left( u-v_H \right) \cdot n_{E}=0$ \ for any $E\in \mathcal{E}_{H}$ and for
any polynomial $s\in \mathbb{P}_{1}(E)$. It means that $\forall T\in
\mathcal{T}_{H}$ and $\forall a\in \mathbb{R}^2$
\begin{equation*}
\int_{T}a\cdot (u-v_H) =\int_{T}\nabla (a\cdot x)\cdot (u-v_H) =\int_{\partial
T}(a\cdot x)n\cdot (u-v_H) =0\,,
\end{equation*}
since $\Div(u-v_H) =0$. In particular,
$$
\sum_{T\in \mathcal{T}_{H}}\int_{\Omega^{\varepsilon }\cap
T}\Pi_{H}(f-\nabla p^{\ast })\cdot (u-v_H) =0
$$
where $\Pi_H$ is the projection on piecewise constant functions, as in Lemma \ref{lem:P0_EF}.
Reminding the last line in (\ref{eq:par_construction}) and using \eqref{eq:P0_EF} and \eqref{eq:poincare perforated broken}, we get
\begin{align*}
&\hspace{-10mm} \sum_{T\in \mathcal{T}_{H}}\int_{\Omega^{\varepsilon }\cap
T}(-\Delta (u-v_{H})+\nabla (p^{\prime }-q_{H}))\cdot (u-v_H) 
\\
=&\quad \sum_{T\in \mathcal{T}_{H}}\int_{\Omega^{\varepsilon }\cap
T}(f-\nabla p^{\ast }-\Pi_{H}(f-\nabla p^{\ast }))\cdot (u-v_H)
\\
\leq&\quad\Vert (f-\nabla p^{\ast })-\Pi_{H}(f-\nabla p^{\ast })\Vert
_{L^2(\Omega)}\,\Vert u-v_H \Vert_{L^2(\Omega^{\varepsilon })}
\\
\leq &\quad C \varepsilon H|(f-\nabla p^{\ast })|_{H^{1}(\Omega)}|u-v_H
|_{H^{1}(\Omega^{\varepsilon })}\,.
\end{align*}

\item The second term in (\ref{eq:decompo_error}) is bounded by
\[
C\varepsilon \left( \sqrt{\varepsilon }+\sqrt{\frac{\varepsilon }{H}}\right)
\left| u-v_H \right|_{H^{1}(\Omega^{\varepsilon })}\,\Vert f-\nabla
p^{\ast }\Vert_{H^{2}(\Omega)\cap C^1(\bar\Omega)}\,,
\]
thanks to Lemma~\ref{lem:brique}.

\item The third term in (\ref{eq:decompo_error}) vanishes. Indeed, on each
edge~$E$, we know from~\eqref{eq:par_construction} that $n\cdot \nabla
v_{H}-q_{H}\,n={a}+b\,n_{E}\psi_{E}$ with some constants $a\in\R^2$, $b\in\R$ and $\int_{E}(
{a}+b\,n_{E}\psi_{E})\cdot (u-v_H) =0$ by construction of $u-v_H$.
\end{itemize}
Collecting all these estimates, we deduce that
\begin{equation*}
|u-v_{H}|_{H^{1}(\Omega^{\varepsilon })}
\leq C\varepsilon \left( \sqrt{\varepsilon }+H+\sqrt{\frac{\varepsilon }{H}
}\right) \Vert f-\nabla p^{\ast }\Vert_{H^{2}(\Omega)\cap C^1(\bar\Omega)} .  
\label{eq:fin_step1}
\end{equation*}
This concludes the estimate for the first term of~\eqref{eq:strang}.

We now turn to the nonconformity error,
\emph{i.e.} the second term in~\eqref{eq:strang}.
Let $v\in \ZH$. We use~\eqref{MsFEMz} and $\Div v=0$ to compute
\begin{eqnarray*}
\int_{\Omega^{\varepsilon}}\nabla (u-u_H):\nabla v
&=&\sum_{T\in \mathcal{T}_{H}}\left( \int_{\Omega^{\varepsilon }\cap
T}\nabla u:\nabla v-\int_{\Omega^{\varepsilon }\cap T}p^{\prime }\Div 
v\right) -\int_{\Omega^{\varepsilon }}f\cdot v \\
&=&\sum_{T\in \mathcal{T}_{H}}\int_{\partial T\cap \Omega^{\varepsilon
}}v\cdot ((\nabla u)n-p^{\prime }n)-\sum_{T\in \mathcal{T}_{H}}\int_{\Omega
_{\varepsilon }\cap T}(f+\Delta u-\nabla p^{\prime })\cdot v \\
&=&\sum_{T\in \mathcal{T}_{H}}\int_{\partial T\cap \Omega^{\varepsilon
}}v\cdot ((\nabla u)n-p^{\prime }n)-\sum_{T\in \mathcal{T}
_{H}}\int_{T}\nabla p^{\ast }\cdot v\,.
\end{eqnarray*}
The first term in the right-hand side above is bounded thanks to Lemma~\ref
{lem:brique} by $C\varepsilon \left( \sqrt{
\varepsilon }+\sqrt{\frac{\varepsilon }{H}}\right) \Vert f-\nabla p^{\ast
}\Vert_{H^{2}(\Omega)\cap C^1(\bar\Omega)}\left| v\right|_{H^{1}}$. To bound the second term, we shall use~$I_H p^\ast\in H^1(\Omega)$
as constructed in Lemma \ref{lem:P1_EF}.
Observe that
\begin{equation}\label{CR3trick1}
\sum_{T\in \mathcal{T}_H}\int_{T}\nabla (I_H p^\ast)\cdot v=\sum_{E\in \mathcal{E}
_{H}}\int_{E}I_H p^\ast\; n_{E}\cdot \lbrack \lbrack v]]=0\,,
\end{equation}
since $I_H p^\ast$ is a polynomial of degree $\le 1$ on each edge $E\in \mathcal{E}_{H}$.
Thus, using Lemmas~\ref{lem:P1_EF} and \ref{lem:poincare perforated broken},
\begin{multline}\label{CR3trick2}
\left\vert \sum_{T\in {{\mathcal T}_{H}}}\int_{T}\nabla p^{\ast }\cdot
v\right\vert
=\left\vert \sum_{T\in {{\mathcal T}_{H}}}\int_{T}\nabla
(p^{\ast }-I_H p^\ast)\cdot v\right\vert
\\
\leq |p^{\ast }-I_H p^\ast|_{H^{1}(\Omega)}\Vert v \Vert_{L^2(\Omega)}\leq
\varepsilon H|p^{\ast }|_{H^{2}(\Omega)}|v |_{H^{1}(\Omega)}\,.
\end{multline}
Finally,
\begin{multline*}
\sup_{v\in \VH{}\setminus \{0\}}\frac{|a (u-u_{H},v)|}{\left| v\right|
_{H^{1}}}
\\
\leq C\varepsilon \left[ \left( \sqrt{\varepsilon }+\sqrt{\frac{
\varepsilon }{H}}+H\right) \Vert f-\nabla p^{\ast }\Vert_{H^{2}(\Omega)\cap C^1(\bar\Omega)}
+H|p^{\ast }|_{H^{2}(\Omega)}\right]\,,
\end{multline*}
which proves the estimate for $( u-u_{H})$ in \eqref{MainErrEst}.

\begin{remark}\label{remCR3trick}
 We have just seen that the nonconformity error has been treated with the help of the trick (\ref{CR3trick1})--(\ref{CR3trick2}) which requires the jumps of the normal component of velocities in $\ZH$ to be orthogonal to polynomials of degree 1 on any edge. This is exactly the motivation to introduce the weights CR3 (\ref{CR3}). This proof would not work with CR2 weights, even if the MsFEM bubbles were added, as suggested in Remark \ref{rembub}.  
\end{remark}

We turn now to the error estimate for pressure. Using operators $\Pi_H$ and $I_H$ from Lemmas \ref{lem:P0_EF} and \ref{lem:P1_EF}, we set $p_H^{\ast} = \Pi_H I_H p^{\ast} \in M_H$, i.e. the
$L^2$-orthogonal projection of $I_H p^{\ast}$ on $M_H$. By interpolation
estimates (\ref{eq:P0_EF}), (\ref{eq:P1_EFl2}) and homogenization bounds
\begin{align}
  \|p_H^{\ast} - p\|_{L^2 (\Omega^{\varepsilon})} &\le \| \Pi_H ( I_H
  p ^{\ast} - p^{\ast})  \|_{L^2 (\Omega)} + \| \Pi_H p ^{\ast} -
  p^{\ast} \|_{L^2 (\Omega)} + \|  p ^{\ast} - p  \|_{L^2
  (\Omega^{\varepsilon})} \notag\\
  & \le \|I_H p ^{\ast} - p^{\ast} \|_{L^2
  (\Omega)} + \| \Pi_H p ^{\ast} - p^{\ast} \|_{L^2 (\Omega)} + \|
  p ^{\ast} - p  \|_{L^2 (\Omega^{\varepsilon})} \notag\\
  &\le C (H|p^{\ast} |_{H^1
  (\Omega)} + \sqrt{\varepsilon} \|f - \nabla p^{\ast} \|_{H^2 (\Omega) \cap C^1 (\bar\Omega)}) .
  \label{pHinterp}
\end{align}
Now, in view of the inf-sup lemmas (\ref{infsupD}) and (\ref{infsupq}), there exists $v_H\in\VH$ such that for any $T\in\TH$
\begin{equation}\label{vHpH}
  \Div v_H=p_H - p_H^{\ast} \text{ on }T\cap\Omega^\eps
  \quad\text{and}\quad
  {|v_H |_{H^1(\Omega)}} \le \frac{C}{\varepsilon} \|p_H - p_H^{\ast} \|_{L^2(\Omega^\eps)} \,. 
\end{equation}
Integration by parts element by element yields
\begin{multline*}
\|p_H - p_H^{\ast} \|_{L^2(\Omega^\eps)}^2=
  \int_{{\Omega}^{{\varepsilon}}}(p_{H}-p_{H}^{{\ast}})\Div v_{H} \\
=-\int_{{\Omega}^{{\varepsilon}}}f{\cdot}v_{H}+\int_{{\Omega}^{{\varepsilon}}}{\nabla}u_{H}:{\nabla}v_{H}-\int_{{\Omega}^{{\varepsilon}}}p_{H}^{{\ast}}\Div v_{H}\\  
=\int_{{\Omega}^{{\varepsilon}}}({\Delta}u-{\nabla}(p^{{\ast}}+p')){\cdot}v_{H}+\int_{{\Omega}^{{\varepsilon}}}{\nabla}u_{H}:{\nabla}v_{H}-\int_{{\Omega}^{{\varepsilon}}}p_{H}^{{\ast}}\Div v_{H}\\
=\int_{{\Omega}^{{\varepsilon}}}{\nabla}(u_{H}-u):{\nabla}v_{H}+\int_{{\Omega}^{{\varepsilon}}}p'\Div v_{H}+\sum_{T{\in}{\mathcal{T}}_{H}}\int_{{\partial}T{\cap}{\Omega}_{{\varepsilon}}}v_{H}{\cdot}(({\nabla}u)n-p'n)\\
    \qquad -\int_{{\Omega}^{{\varepsilon}}}{\nabla}(p^{{\ast}}-I_{H}p^{{\ast}}){\cdot}v_{H}-\int_{{\Omega}^{{\varepsilon}}}{\nabla}I_{H}p^{{\ast}}{\cdot}v_{H}-\int_{{\Omega}^{{\varepsilon}}}p_{H}^{{\ast}}\Div v_{H}\\
\end{multline*}
In fact, the last two terms above cancel each other. Indeed,
\begin{multline*}
 -\int_{{\Omega}^{{\varepsilon}}}{\nabla}I_{H}p^{{\ast}}{\cdot}v_{H}-\int_{{\Omega}^{{\varepsilon}}}p_{H}^{{\ast}}\Div v_{H}
=-\sum_{E{\in}{\mathcal{E}}_{H}}\int_{E}I_Hp^*[[n{\cdot}v_{H}]]\\
    + \int_{{\Omega}^{\varepsilon}}(I-\Pi_H)(I_{H}p^{{\ast}}) \Div v_{H}=0.
\end{multline*}
This is zero since $[[n\cdot v_H]]$ is orthogonal to the polynomials of degree $\le 1$ on the edges and
$\Div{v}_H \in M_H$. 

We can now apply the already proven upper bound for the velocity error $(u-u_H)$ in (\ref{MainErrEst}), 
 Lemma~\ref{lem:brique}, and bound (\ref{CR3trick2}) to conclude
\begin{multline*}
 \|p_H - p_H^{\ast} \|_{L^2(\Omega^\eps)}^2 \\ \leq C \varepsilon \left[
   \left( \sqrt{\varepsilon} + \sqrt{\frac{\varepsilon}{H}} + H \right) \|f -
   \nabla p^{\ast} \|_{H^2 (\Omega) \cap C^1 (\bar\Omega)} + H|p^{\ast} |_{H^2 (\Omega)} \right]
     \|v_H\|_{H^1(\Omega)}
   \\+  \|p'\|_{L^2(\Omega^\eps)}\|\Div v_H\|_{L^2(\Omega)}
   \,, 
\end{multline*}
Recalling the properties of $v_H$ (\ref{vHpH}) and the homogenization estimate (\ref{errP}) for $p' = p - p^{\ast}$, this entails
\[ \|p_H - p_H^{\ast} \|_{L^2} \leq C \left[
   \left( \sqrt{\varepsilon} + \sqrt{\frac{\varepsilon}{H}} + H \right) \|f -
   \nabla p^{\ast} \|_{H^2 (\Omega) \cap C^1 (\bar\Omega)} + H|p^{\ast} |_{H^2 (\Omega)} \right]
   \,, \]
which in combination with (\ref{pHinterp}) gives the error estimate for
pressure in (\ref{MainErrEst}) by the triangle inequality.

\section{Numerical results}
\label{sec:numres}

In this section we show some results of numerical computations, for both variants of our
method, CR$_2$ and CR$_3$, cf. (\ref{CR2}) and (\ref{CR3}). All calculations are performed in \texttt{FreeFem++}
\cite{freefem}.

\subsection{Implementation details}
The Crouzeix-Raviart MsFEM as presented so far relies on the exact solutions of the local problems in the construction of the basis functions. In practice, these problems should be discretized on a mesh sufficiently fine to resolve the geometry of obstacles. To avoid complex and ad-hoc grid generation methods when solving (\ref{genP})--(\ref{genB}) in $\Omega^\epsilon$ we replace it with the penalized problem, cf. \cite{StokesPart1}. To calculate both the basis functions and the reference solutions, we use the P1-P1 FEM on the uniform Cartesian grid $\mathcal{T}_h$ of step $h<<H$. As is well known, this choice of velocity and pressure spaces requires some stabilization which weakens the condition $\nabla \cdot \vec{u} = 0$. The simplest way to achieve this is by perturbing the incompressibility constraint with a pressure Laplacian term, see \cite{brezzipitkaranta} and \cite{StokesPart1}. The reference solution is calculated on the global mesh of the same size as that for the MsFEM basis functions.

\subsection{Test case with periodic holes}\label{sec:test1}
For our first test case we choose $\Omega=(0,1)^2$ and $B^\eps$ as the set of discs of radius $\eps/4$ placed periodically on a regular grid of period $\eps=\frac{1}{128}$. We solve Stokes equations (\ref{genP})--(\ref{genB}) on $\Omega^\eps=\Omega\setminus B^\eps$ with $f=\left(\begin{array}{c}
-(x_2-1/2) \\ x_1-1/2 
\end{array}\right)$. The fine regular Cartesian mesh with $h=\frac{1}{1536}$ is used to compute both the reference solution and the MsFEM basis functions. The results are reported in \figref{fig:test1}. The error curves for the velocity are compatible with theoretical estimate (\ref{MainErrEst}). We observe indeed a decrease of the error with refinement in $H$ for $H>>\eps$ and a plateau when $H\sim\eps$. As expected, CR3 variant of the method produces a much more accurate solution than CR2 one. The error curves for $p-p_H$ (with $p_H$ being the piecewise constant approximation to the exact pressure $p$) are qualitatively even better than the theoretical bound with respect to the mesh refinement. Apart from the piecewise approximation $p_H$ we also report on an ``oscillating'' reconstruction $p_H+\pi_H(u_H)$ as suggested by (\ref{XHdef1}), i.e. reusing the local pressure contributions $\pi_{E,i}$ associated to the velocity basis functions $\Phi_{E,i}$. One could hope that adding $\pi_H(u_H)$ would improve the accuracy of $p_H$. The numerical experiments do not support this conjecture: in fact,  adding $\pi_H(u_H)$ can even deteriorate the accuracy. 

\begin{figure}[tbp]
    \centering
    \includegraphics[width=0.45\textwidth]{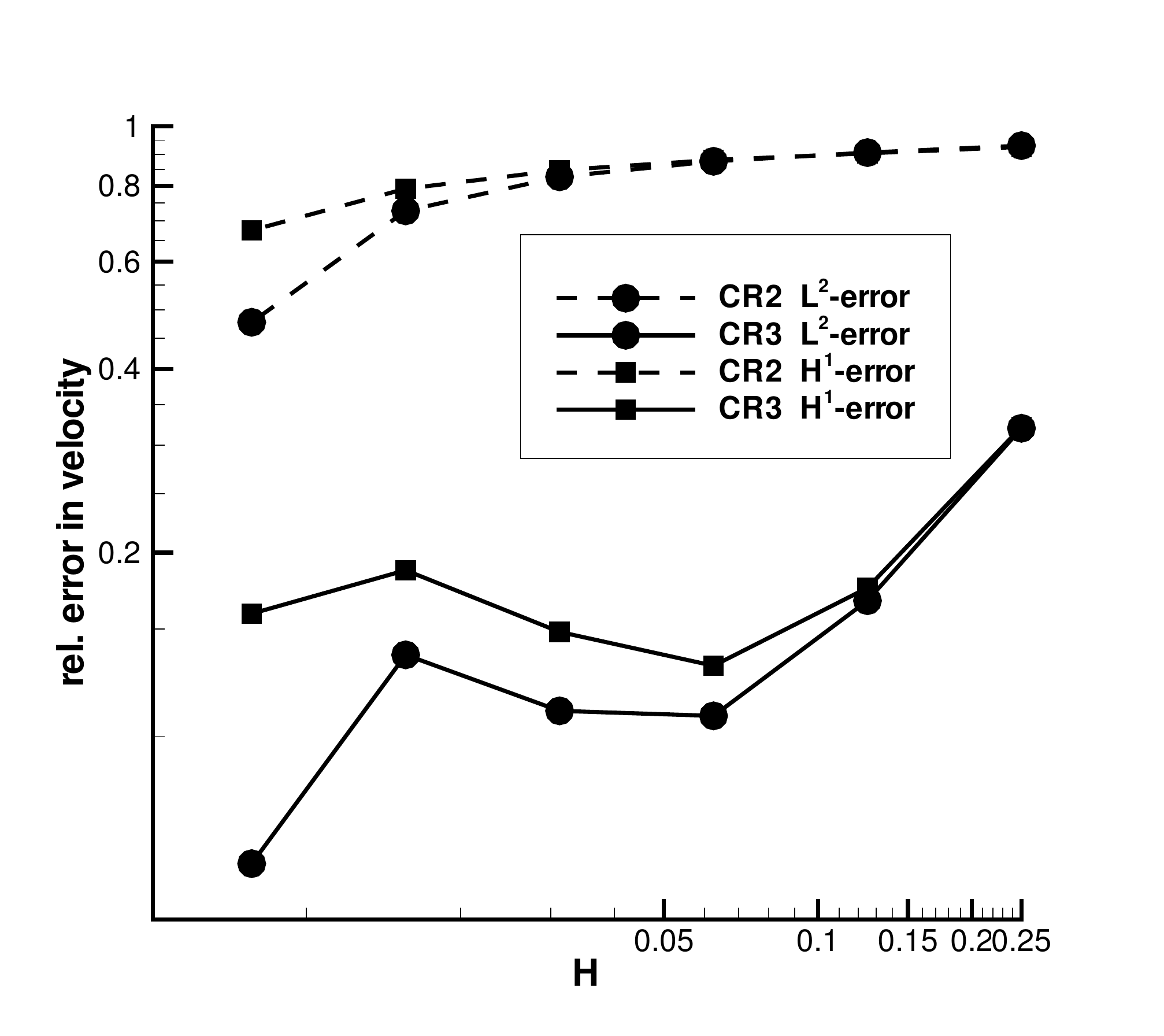}
    \quad
    \includegraphics[width=0.45\textwidth]{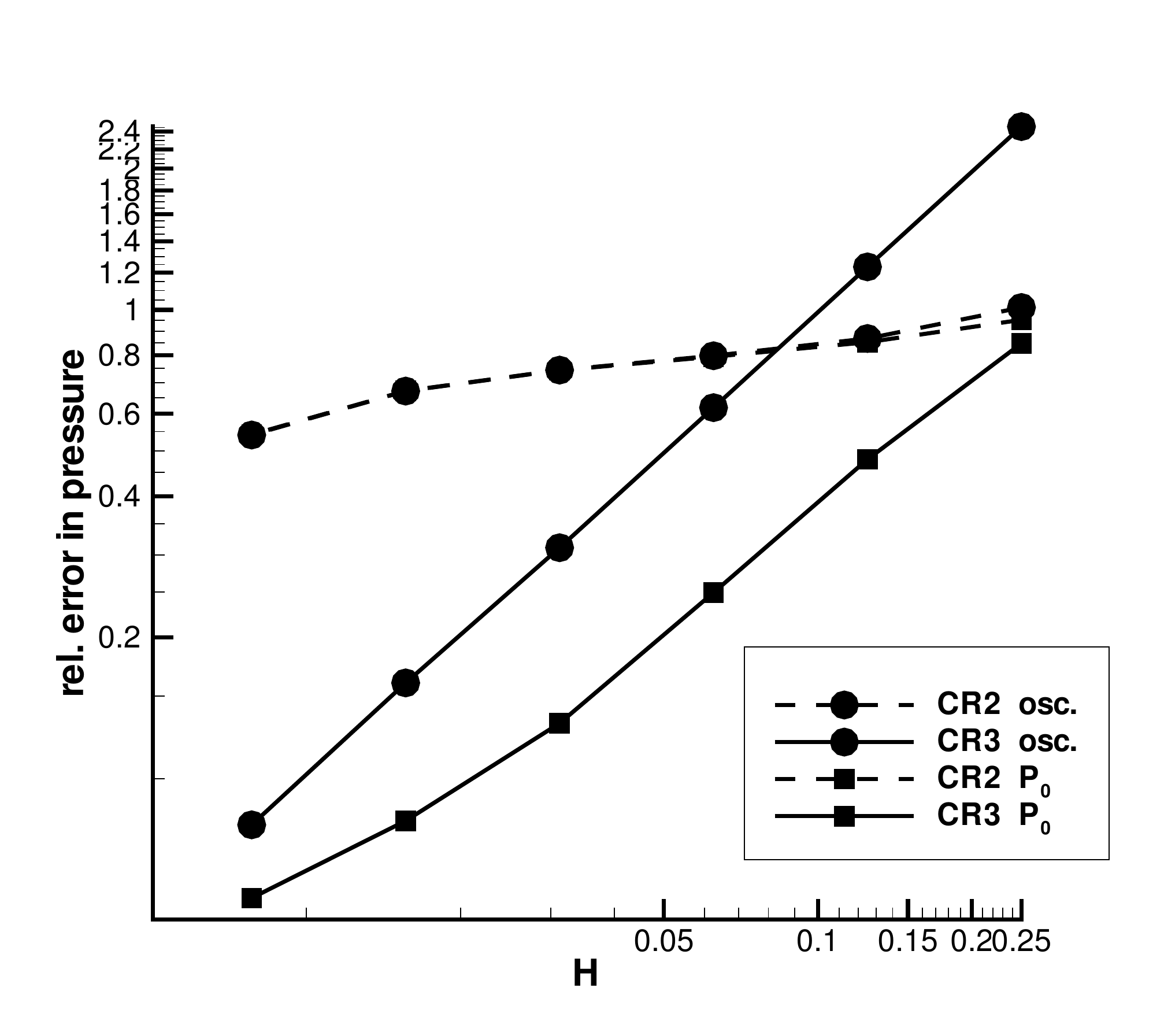}
    \caption{Test case of Section \ref{sec:test1}. Left: the relative error in velocity $u-u_H$ in $L^2$ and $H^1$ norms. Right: the relative error in pressure computed either as $p-p_H$ (denoted P$_0$) or as $p-p_H-\pi_H(u_H)$ (denoted osc.). The mesh size $H$ varies from $\frac{1}{4}$ down to $\frac{1}{128}$ with $\eps=\frac{1}{128}$.}
    \label{fig:test1}
\end{figure}

\subsection{Channel flow}\label{sec:channel} We turn now to a more realistic test case: a flow in a rectangle $\Omega=(0,1.5)\times (0,1)$ with several obstacles $B^\eps$ inside with parabolic velocity profile prescribed on the vertical edges. We solve thus (\ref{genP})--(\ref{genD})--(\ref{genBg}) with $f=0$ and the boundary conditions $u=4x_2(1-x_2)e_1$ on $\partial\Omega$.   The adaptation of our method in view of non-homogeneous boundary conditions is presented in Remark \ref{BCg}.

The obstacles $B^\eps$ are presented in \figref{fig:channel}, top left.  They are constructed as follows: we define the perforation pattern $B$ on the reference cell $Y=(0,1)^2$ as $\mathbbm{1}_{|x_1+x_2-1|<\frac 12} \mathbbm{1}_{|x_1-x_2|<\frac 14}$, then we shrink $Y$ by factor $\eps$ (with $\eps=0.15$ in \figref{fig:channel}) and repeat it periodically. Finally, we eliminate the holes outside the rectangle $(0.25,1.4) \times (0,1)$ in order to leave a little space between the region where the flow is perturbed by the obstacle and the zone of free Poiseuille flow to the left and to the right. The reference solution and MsFEM CR2 and CR3 solutions (namely the $u_1$ velocity component) are reported in  \figref{fig:channel}. We observe that the CR3 variant captures the essential features of the solution even on a very coarse $6\times 4$ mesh, while the solution produced by the CR2 variant is completely wrong.   

\begin{figure}[tbp]
    \centering
    \includegraphics[width=0.49\textwidth]{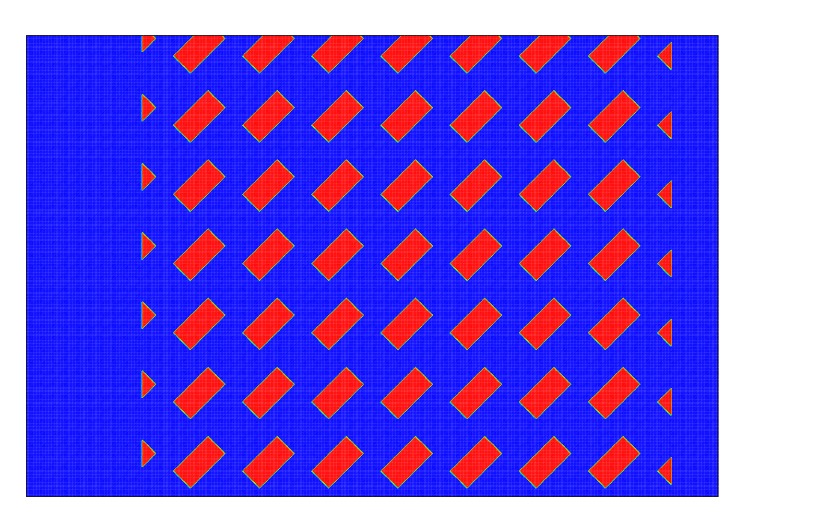}
    \includegraphics[width=0.49\textwidth]{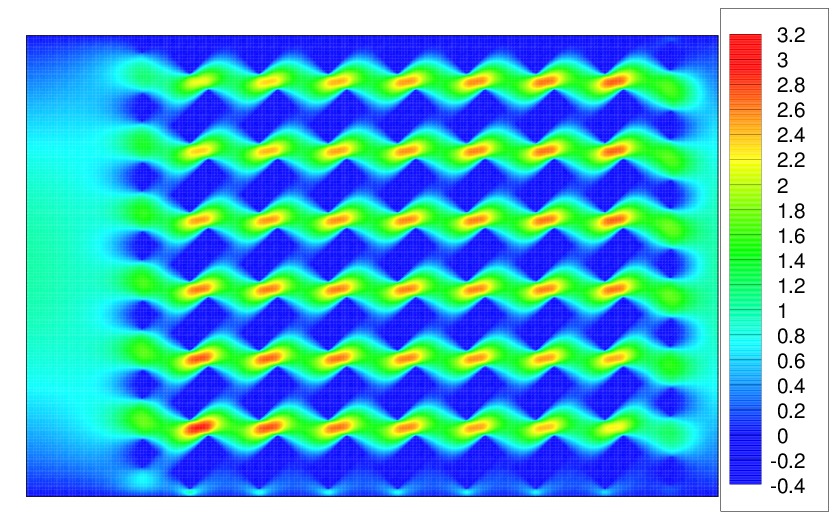}

    \includegraphics[width=0.48\textwidth]{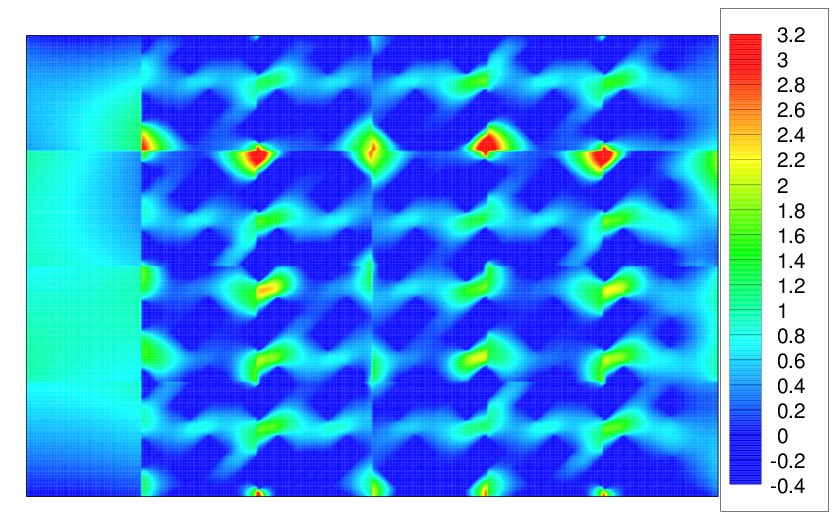}
    \includegraphics[width=0.48\textwidth]{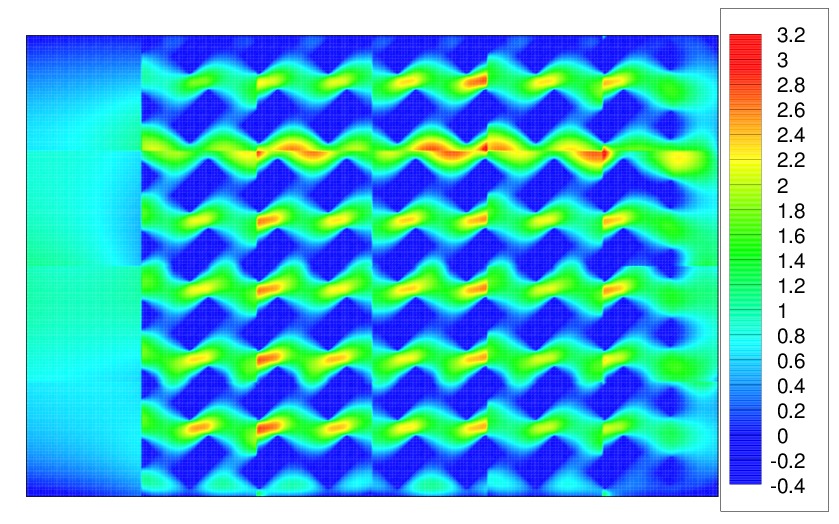}
    \caption{Test case of Section \ref{sec:channel}: channel flow. Top left : fluid domain $\Omega^\eps$ in blue (with obstacles in red). Top right: the reference solution on the $750\times 500$ grid. Bottom: the MsFEM solution on $6\times 4$ grid; CR2 on the left, CR3 on the right. The $u_1$ velocity component is represented on the contour plots.}
    \label{fig:channel}
\end{figure}

\bibliographystyle{siam}
\bibliography{crmsfem}

\end{document}